\documentclass[smallextended,envcountsect]{svjour3}
\usepackage{graphicx,color,cite}
\usepackage{mathrsfs}
\usepackage{amsmath, amscd, amsfonts, amssymb, graphicx, color}
\usepackage[bookmarksnumbered, colorlinks, plainpages]{hyperref}
\usepackage{mathptmx}   
\usepackage[colorinlistoftodos]{todonotes}
\usepackage[left=4cm, right=4cm, top=3.5cm, bottom=3.5cm]{geometry}

\usepackage{latexsym}

\if{
\newtheorem{theorem}{Theorem}[section]
\newtheorem{definition}{Definition}[section]
\newtheorem{lemma}{Lemma}[section]

\newtheorem{remark}{Remark}[section]

}\fi

\def\epi{\mathop{\rm epi\,}}

\renewcommand\theenumi{\roman{enumi}}

\newcounter{mycount}

\smartqed

\usepackage{etex}

\usepackage{mathtools}

\usepackage{enumitem}
\renewcommand\theenumi{(\roman{enumi})}
\renewcommand\theenumii{(\alph{enumii})}

\usepackage{csquotes}

\usepackage{todonotes}

\makeatletter
\let\orgdescriptionlabel\descriptionlabel
\renewcommand*{\descriptionlabel}[1]{
 \let\orglabel\label
 \let\label\@gobble
 \phantomsection
 \edef\@currentlabel{#1}
 \let\label\orglabel
 \orgdescriptionlabel{#1}
}
\def\th@plain{
 \thm@notefont{}
 \itshape
}
\def\th@definition{
 \thm@notefont{}
 \normalfont
}

\g@addto@macro\th@definition{\thm@headpunct{}}
\g@addto@macro\th@plain{\thm@headpunct{}}
\makeatother

\usepackage{subfig}

\usepackage[final]{showkeys}

\usepackage{etoolbox}

\usepackage{mathrsfs}

\usepackage[titletoc]{appendix}
\usepackage[doc]{optional}

\usepackage{soul}

\usepackage{colortbl,booktabs,
multirow}
\usepackage{xcolor}

\usepackage{cancel}

\usepackage{empheq}

\definecolor{myblue}{rgb}{.8, .8, 1}

\usepackage{hyperref}
\hypersetup{
colorlinks=true,
linkcolor=blue,
citecolor=red,
filecolor=magenta,
urlcolor=cyan
}

\usepackage[
open,
openlevel=2,
atend,
numbered
]{bookmark}

\usepackage[capitalize, nameinlink, noabbrev]{cleveref}
\crefname{equation}{}{}
\crefname{chapter}{Chapter}{Chapters}
\crefname{item}{item}{items}
\crefname{figure}{Figure}{Figures}
\crefname{theorem}{Theorem}{Theorems}
\crefname{lemma}{Lemma}{Lemmas}
\crefname{proposition}{Proposition}{Propositions}
\crefname{corollary}{Corollary}{Corollarys}
\crefname{definition}{Definition}{Definitions}
\crefname{fact}{Fact}{Facts}
\crefname{example}{Example}{Examples}
\crefname{algorithm}{Algorithm}{Algorithms}
\crefname{remark}{Remark}{Remarks}
\crefname{note}{Note}{Notes}
\crefname{notation}{Notation}{Notations}
\crefname{case}{Case}{Cases}
\crefname{exercise}{Exercise}{Exercises}
\crefname{question}{Question}{Questions}
\crefname{claim}{Claim}{Claims}
\crefname{enumi}{}{}
\usepackage{pgf}

\usepackage{array}
\usepackage{tabu}

\allowdisplaybreaks

\numberwithin{equation}{section}

\renewcommand\theenumi{(\roman{enumi})}
\renewcommand\theenumii{(\alph{enumii})}

\spnewtheorem*{Proof}{Proof.}{\bf}{\rm}

\def\bd{\mathop{\rm bd\,}}

\begin{document}

\title{Metric Subregularity  of Multifunctions and Applications to Characterizations of Asplund Spaces\thanks{Research  of the first author was supported 
by Science and Technology Project of Hebei Education Department (No. ZD2022037) and the Natural Science Foundation of Hebei Province (A2022201002).}}

\titlerunning{Metric Subegularity  of Multifunctions and Applications to Characterizations of Asplund Spaces}

\author{Zhou Wei  \and Michel Th\'era \and Jen-Chih Yao}

\institute{Zhou Wei\at Hebei Key Laboratory of Machine Learning and Computational Intelligence \& College of Mathematics and Information Science, Hebei University, Baoding, 071002, China\\ \email{weizhou@hbu.edu.cn}\\
Michel Th\'era \at XLIM UMR-CNRS 7252, Universit\'e de Limoges, Limoges, France\\\email{michel.thera@unilim.fr}\\
Jen-Chih Yao \at Research Center for Interneural Computing, China Medical University Hospital,
China Medical University, Taichung, Taiwan.\\ \email{yaojc@mail.cmu.edu.tw}
}

\date{Received: date / Accepted: date}

\maketitle

\begin{abstract}

In this paper, we investigate metric subregularity of {multifunctions} between Asplund spaces.  
Using Mordukhovich normal cones and coderivatives, we introduce the limiting Basic Constraint Qualification $\textbf{(BCQ)}$ associated with a given {multifunction}.  
This $\textbf{BCQ}$ provides necessary dual conditions for the metric subregularity of  {multifunctions} in the Asplund  {space} setting.  

Furthermore, we establish characterizations of Asplund spaces in terms of the  {limiting} $\textbf{BCQ}$ condition implied by metric subregularity.  
By employing Fr\'echet normal cones and coderivatives, we derive necessary dual conditions for metric subregularity expressed as fuzzy inclusions, and we also obtain characterizations of Asplund spaces via these fuzzy inclusions.  

As an application, we examine  metric subregularity of  {the conic inequality} defined by a vector-valued function and a closed (not necessarily convex) cone with a nontrivial recession cone.  
By using Mordukhovich and Fr\'echet subdifferentials relative to the given cone, we establish necessary dual conditions for the metric subregularity of such inequalities in Asplund spaces.  
The results based on Mordukhovich subdifferentials characterize Asplund spaces, while those based on Fr\'echet subdifferentials yield necessary or sufficient conditions for Asplund spaces.  
These conditions recover, as special cases, the known error-bound results for inequalities defined by extended-real-valued functions on Asplund spaces.  

Overall, this work highlights that the validity of necessary conditions formulated via normal cones and subdifferentials for error bounds of convex or nonconvex inequalities depends crucially on the Asplund property of the underlying space.


\end{abstract}

\section{Introduction}
The main objective of this paper is to investigate the metric subregularity of multifunctions between Asplund spaces and to derive necessary dual conditions using coderivatives and normal cones. Based on these results, we establish new characterizations of Asplund spaces in terms of such dual conditions.

As highlighted in Ioffe's monograph \cite{Ioffe2017}, the concept of metric subregularity---introduced as ``regularity at a point'' in \cite{Ioffe2000} and later popularized under this name by Dontchev and Rockafellar \cite{DonRoc2004}---has emerged as one of the most fundamental notions in variational analysis and optimization.

Given a multifunction $F:\mathbb{X}\rightrightarrows\mathbb{Y}$ between Banach spaces $\mathbb{X},\mathbb{Y}$, recall from \cite{DonRoc2004} that $F$ is said to be \emph{metrically subregular} at $(\bar x,\bar y)\in{\rm gph}(F):=\{(u,v)\in\mathbb{X}\times\mathbb{Y}: v\in F(u)\}$ if there exists $\tau\in (0,+\infty)$ such that 
\begin{equation}\label{1.1-250624}
	{\bf d}(x, F^{-1}(\bar y))\leq \tau\, {\bf d}(\bar y, F(x))  
	\quad \text{for all $x$ close to $\bar x$}.
\end{equation}
This property provides an upper bound for the distance from a point $x$ to the solution set $F^{-1}(\bar y):=\{x \in \mathbb{X}: \bar y \in F(x)\}$.

The study of metric subregularity originates from the celebrated Lyusternik-Graves theorem \cite{Lyusternik1934,Graves1950}. Since then, it has been extensively developed in numerous monographs, including \cite{DonRoc09,Mordukhovich,Ioffe2017,Pen13,thibault}, and investigated under various names (see, e.g., \cite{Aze06,DL,HenrionJourani2002-SIAM,HenrionJouraniOutrata2002-SIAM,HenrionOutrata2004-MP,Ioffe2000,IoffeOutrata2008-SVA,ZhengNg2010-SIAM,ZhengNg2012-NA}).

As observed by Henrion and Outrata \cite{HenOut2005}, $F$ is metrically subregular at $(\bar x ,\bar y)$ if and only if its inverse $M:=F^{-1}$ is \emph{calm} at $(\bar y, \bar x)$; that is, there exists $\tau>0$ such that
\begin{equation}\label{1.3}
	{\bf d}(x, M(\bar y))\leq \tau\, \|y-\bar y\|
	\quad \text{for all $x$ close to $\bar x$}.
\end{equation}

The first fundamental finite-dimensional result of type \eqref{1.1-250624} is due to Hoffman \cite{Hoffman1952}, who proved that for a linear system $Ax\leq b$, the distance from any point $x$ to the solution set is bounded above by a constant times the norm of the residual $\max \{Ax-b, 0\}$.

Given a proper lower semicontinuous function $f: \mathbb{X}\to\mathbb{R}\cup\{+\infty\}$, define $F(x):=[f(x),+\infty)$ for $x\in \mathbb{X}$ and let 
\[
\bar x\in  \mathbf{S}_f:=\{x\in \mathbb{X}: f(x)\leq 0\}.
\]
Then the metric subregularity of $F$ at $(\bar x,0)$ reduces to the error bound inequality
\begin{equation}\label{1.1}
	{\bf d}(x,\mathbf{S}_f)\leq \tau f_+(x)\quad 
	\text{for all $x$ close to $\bar x$},
\end{equation}
where $f_+(x):=\max\{f(x),0\}$. Such Hoffman-type estimates, now widely known as \emph{error bounds}, form an essential part of modern optimization theory and have been extensively studied (see, e.g., \cite{AC1988,BD2,ioffe-JAMS-1,ioffe-JAMS-2,K2,Pen13,Robinson1975,WTY2024-SVVA}).

It is worth noting that Lewis and Pang \cite{LewisPang1998} studied dual necessary conditions, expressed in terms of normal cones and subdifferentials, for the error bounds of convex inequality systems. In particular, they proved that the basic constraint qualification ($\mathbf{BCQ}$) for a convex inequality constitutes a necessary condition.

For a continuous convex function $f$ defined on $\mathbb{X}$, recall that the convex inequality $ f(x)\leq 0$ is said to satisfy $\mathbf{BCQ}$ at $x\in \bd(\mathbf{S}_f)$ (the boundary of $\mathbf{S}_f$) if
\begin{equation}\label{1.2}
	\mathbf{N}(\mathbf{S}_f,x)=[0,+\infty)\partial f(x),
\end{equation}
where $\mathbf{N}(\mathbf{S}_f,x)$ and $\partial f(x)$ denote, respectively, the normal cone and the subdifferential in convex analysis. The condition $\mathbf{BCQ}$ is fundamental in convex analysis and closely related to important notions in optimization and approximation, such as the Slater condition, the strong conical hull intersection property (strong CHIP), the Abadie constraint qualification (ACQ), and KKT optimality conditions (see, e.g., \cite{H-UL1993,De,Li,LNS2000,LNP2008,FLN2010}).  

Zheng and Ng \cite{ZN2004} extended $\mathbf{BCQ}$ to general convex inequalities using singular subdifferentials, and later \cite{WY} further generalized it to nonconvex settings via Clarke and Fr\'echet subdifferentials.
 {Further,  Zheng and Ng \cite{ZN2007}   generalized the $\mathbf{BCQ}$ to the convex multifunctions case and used it to provide dual characterizations in terms of normal cones and coderivatives for the metric subregularity of convex multifunctions}.

In this paper, we introduce a limiting $\mathbf{BCQ}$ for closed multifunctions, formulated in terms of Mordukhovich normal cones and coderivatives. Our main result establishes that this limiting $\mathbf{BCQ}$ is a necessary condition for the metric subregularity of multifunctions between Asplund spaces. Moreover, we obtain characterizations of Asplund spaces based on $\mathbf{BCQ}$ conditions implied by metric subregularity (see \cref{th3.1}).  

Using Fr\'echet normal cones and coderivatives, we also show that metric subregularity of closed multifunctions between Asplund spaces implies certain fuzzy inclusions, which yield further characterizations of Asplund spaces (see \cref{theorem3.2}).  

As an application, we study the metric subregularity of a conic inequality defined by a vector-valued function and a closed (not necessarily convex) cone with nontrivial recession cone. By employing Mordukhovich and Fr\'echet subdifferentials relative to the cone, we show that exact and fuzzy inclusions are, respectively, implied by metric subregularity in Asplund spaces.  

The validity of exact inclusions leads to characterizations of Asplund spaces, whereas the fuzzy inclusions yield necessary or sufficient conditions for Asplund spaces (see \cref{th4.1,theorem4.2}).

The paper is organized as follows. In section 2, we present definitions and preliminary
results used throughout the paper. Section 3 is devoted to the main results on metric subregularity of closed multifunctions between Asplund spaces and some new characterizations for Asplund spaces. As an application, the metric subregularity of conic inequalities is studied in section 4. The conclusions of this paper is presented in section 5.
 \medskip
 
\section{Notation, definitions and  background material}


Let $\mathbb{X}$ and $\mathbb{Y}$ be  Banach spaces with duals $\mathbb{X}^*$ and $\mathbb{Y}^*$, respectively.  The duality pairing between a Banach space $\mathbb{X}$ and its dual $\mathbb{X}^*$ is denoted by \(\langle \cdot, \cdot \rangle\).  We use $\|\cdot\|$ to denote the norm in $\mathbb{X}$. The closed unit ball of $\mathbb{X}$ is denoted by $\mathbf{B}_{\mathbb{X}}$. Let $\mathbf{B}_{\mathbb{X}^*}$ and $\mathbf{S}_{\mathbb{X}^*}$ the closed unit ball and the sphere of $\mathbb{X}^*$, respectively. For any $x\in\mathbb{X}$ and $\delta>0$, we denote by ${\bf B}(x,\delta)$ the open ball with center $x$ and radius $\delta$.
	 
	For a set $A \subseteq \mathbb{X}$, we write $\overline{A}$ for its closure, $\mathrm{int}(A)$ for its interior, and $\mathrm{conv}(A)$ for its convex hull. The distance from a point $x \in \mathbb{X}$ to a set $A$ is  
	\[
	\mathbf{d}(x,A) := \inf_{a \in A} \|x-a\|,
	\]
	with the convention $\mathbf{d}(x,\emptyset) := +\infty$.

Given a nonempty set $S\subseteq \mathbb{X}$ (not necessarily convex)  the recession cone of $S$ is defined to be the set
$$S^{\infty}:=\{u\in \mathbb{X} : s+tu\in S, \forall s\in S, \forall t\geq 0 \}.$$ When $S$ is  convex, it is { known }from \cite{Zalinescu}  that 
$$S^{\infty}=\bigcap_{t>0}t(S-s) $$
for each $s\in S$.

Given a multifunction $F: \mathbb{X} \rightrightarrows \mathbb{Y}$ between $\mathbb{X} $ and $ \mathbb{Y}$, we denote by   
\[
{{\rm gph}(F)} := \{(x,y) \in \mathbb{X} \times \mathbb{Y} \;:\; y \in F(x)\},
\]
the {\it graph} of $F$. Its domain is ${\rm dom}(F):= \{x \in \mathbb{X} : F(x) \neq \emptyset\}$,  
and its inverse is defined by 
$$F^{-1}(y) := \{x \in \mathbb{X} : y \in F(x)\}, \ \forall y\in\mathbb{Y}.$$

{Given  a multifunction $\Phi:\mathbb{X}\rightrightarrows\mathbb{X}^*$ between  $\mathbb{X}$ and $\mathbb{X}^*$},  
the \emph{sequential Painlev\'{e}--Kuratowski outer (upper) limit} of $\Phi$ at $x$ is defined by
\begin{equation*}
\mathop{\rm Limsup}_{y\to x}\Phi(y)
:= \left\{ x^* \in \mathbb{X}^* \ \middle| \
\begin{array}{l}
\exists\, (x_n) \to x,\ \exists\, (x_n^*) \stackrel{w^*}{\longrightarrow} x^* \ \text{with} \\
x_n^* \in \Phi(x_n) \ \text{for all } n\in\mathbb{N}
\end{array}
\right\}.
\end{equation*}

\medskip
\noindent
\textbf{Normal cones.}

For a closed set $\Omega \subseteq \mathbb{X}$ and $\bar{x} \in \Omega$:
\begin{itemize}
    \item The \emph{Fr\'echet normal cone} to $\Omega$ at $\bar{x}$ is  
    \[
    \widehat{\bf N}(\Omega, \bar{x}) := \left\{ x^* \in \mathbb{X}^* \; : 
    \limsup_{x \xrightarrow{\Omega} \bar{x}} \frac{\langle x^*, x - \bar{x} \rangle}{\|x - \bar{x}\|} \le 0 \right\},
    \]
     {where $x \xrightarrow{\Omega} \bar{x}$ means that $x\rightarrow\bar x$ with $x\in\Omega$.}
    \item 
   For $\varepsilon\geq 0$, the set of $\varepsilon$-{normals}  to $\Omega$ at $\bar x$ is
$$
\widehat{\mathbf{N}}_{\varepsilon}(\Omega, \bar x):= {\left\{x^*\in \mathbb{X}^*:\limsup\limits_{x \xrightarrow{\Omega} \bar x}\frac{\langle x^*, x-\bar x\rangle}{\|x-\bar x\|}\leq \varepsilon\right\}}.
$$
 When $\varepsilon=0$, $\widehat{\mathbf{N}}_{0}(\bar x,\Omega)$ coincides with $ {\widehat{\bf N}(\Omega, \bar{x})}$.
      \item The \emph{Mordukhovich (limiting) normal cone} is  
\begin{equation}\label{2.1}
   {{\bf N}(\Omega, \bar{x}) := \mathop{\rm Limsup}_{x\stackrel{\Omega}\longrightarrow \bar x, \varepsilon \downarrow 0}\widehat{\mathbf{N}}_{\varepsilon} {(\Omega, x)}}.
\end{equation}
   Thus, $x^*\in {\bf N}(\Omega,\bar x)$ if and only if there exists a sequence $\{(\varepsilon_k,x_k,x_k^*)\}$ in $(0,+\infty)\times\Omega\times \mathbb{X}^*$ such that $\varepsilon_k\rightarrow 0^+, x_k\stackrel{w^*}{\rightarrow}x^*$ and $x_k^*\in\widehat{\mathbf{N}}_{\varepsilon} {(\Omega, x_k)}$ for all $n$.
\end{itemize}
It is known from  \cite{C} and \cite{Mordukhovich} that
$$ {\widehat{\bf N}(\Omega, \bar x)\subseteq {\bf N}(\Omega, \bar x)}.$$
If $ {\Omega}$ is convex, Fr\'echet and limiting normal cones coincide and reduce to the normal cone in the sense of convex analysis; that is
$$
 {\widehat{\bf N}(\Omega, \bar x) = {\bf N}(\Omega, \bar x)=\{x^*\in \mathbb{X} ^*:\;\langle
x^*,x-\bar x\rangle\leq 0\;\; {\rm for\ all} \ x\in \Omega\}}.
$$

\medskip

\noindent
\textbf{Asplund spaces.}  

Recall that a Banach space is called an Asplund space if every continuous convex function on $\mathbb{X}$  is Fr\'echet differentiable at each point of a dense subset of $\mathbb{X}$ (see \cite{Phelps} for definitions and their equivalences). {It is known that a Banach space is an Asplund  space if and only if every separable closed subspace has a separable dual. In particular, every reflexive space is an Asplund space}. For the case when $\mathbb{X}$  is an Asplund space, Mordukhovich and Shao \cite{MS} have proved that
\begin{equation*}
 {	{\bf N}(\Omega, \bar x)=\mathop{\rm Limsup}_{x\xrightarrow{\Omega}
		\bar x}\widehat{\bf N}(\Omega, x)}.
\end{equation*}
This means that $x^*\in  {	{\bf N}(\Omega, \bar x)}$ if and only if there exist $x_n {\xrightarrow{\Omega}
\bar x}$ and $x^*_n\xrightarrow{w^*} x^*$ such that $x_n^*\in  {\widehat{\bf N}(\Omega, x_n)}$ for all $n$.

\medskip

\noindent
\textbf{Subdifferentials.}
  
Let $\varphi : \mathbb{X} \to \mathbb{R} \cup \{+\infty\}$ be a proper lower semicontinuous function.  {We denote by 
$${\rm epi}(\varphi):= \{(u,r) \in \mathbb{X} \times \mathbb{R}: \varphi(u) \leq r\}$$
the {\it epigraph} of $\varphi$}. For $x \in \mathrm{dom}(\varphi):= \{y \in \mathbb{X} : \varphi(y) < +\infty\}$, we denote by $\widehat{\partial} \varphi(x)$ and $\partial \varphi(x)$ the \emph{Fr\'echet} and \emph{limiting} (Mordukhovich) subdifferentials of $\varphi$ at $x$ and they are given by
\[\widehat{\partial} \varphi(x) := \left\{x^* \in \mathbb{X}^* : (x^*, -1) \in \widehat{\bf N}\big(\mathrm{epi}(\varphi), (x, \varphi(x))\big) \right\},
\]
\[
\partial \varphi(x) := \left\{x^* \in \mathbb{X}^* : (x^*, -1) \in {\bf N}\big(\mathrm{epi}(\varphi), (x, \varphi(x))\big) \right\}.
\]


It is straightforward to verify that  
\[
\widehat{\partial} \varphi(x) =
\left\{x^* \in \mathbb{X}^* : \liminf_{y \to x} 
\frac{\varphi(y) - \varphi(x) - \langle x^*, y - x \rangle}{\|y - x\|} \ge 0 \right\}.
\]  

When $\mathbb{X}$ is an \emph{Asplund space}, Mordukhovich and Shao \cite{MS} proved that  
\[
\partial \varphi(x) = \mathop{\mathrm{Limsup}}_{y \xrightarrow{\varphi} x} \, \widehat{\partial} \varphi(y),
\]
where $y \xrightarrow{\varphi} x$ means $y \to x$ with $\varphi(y) \to \varphi(x)$.  
Equivalently, $x^* \in \partial \varphi(x)$ if and only if there exist sequences $x_n \xrightarrow{\varphi} x$ and $x_n^* \xrightarrow{w^*} x^*$ with $x_n^* \in \widehat{\partial} \varphi(x_n)$ for all $n$.

\medskip
\noindent
%
%

\textbf{Coderivatives.}  

For a multifunction $F: \mathbb{X} \rightrightarrows \mathbb{Y}$ and a point $(\bar{x}, \bar{y}) \in \mathrm{gph}(F)$, the \emph{Mordukhovich coderivative} of $F$ at $(\bar{x}, \bar{y})$ is defined by  
\[
D^*F(\bar{x},\bar{y})(y^*) := 
\left\{x^* \in \mathbb{X}^* : (x^*, -y^*) \in {\bf N}\big(\mathrm{gph}(F), (\bar{x}, \bar{y})\big)\right\}, 
\quad y^* \in \mathbb{Y}^*.
\]  
Similarly, the \emph{Fr\'echet coderivative} of $F$ at $(\bar{x}, \bar{y})$ is given by  
\[
\widehat{D}^*F(\bar{x},\bar{y})(y^*) := 
\left\{x^* \in \mathbb{X}^* : (x^*, -y^*) \in \widehat{\bf N}\big(\mathrm{gph}(F), (\bar{x}, \bar{y})\big)\right\}, 
\quad y^* \in \mathbb{Y}^*.
\]  
\medskip

\medskip
We begin by recalling a fundamental result due to Fabian~\cite[Theorem~3]{F}, 
often referred to as the \emph{fuzzy sum rule} together with the density property 
for Fr\'echet subgradients of lower semicontinuous functions in Asplund spaces. 
For further details, the reader may consult 
\cite[Proposition~2.7]{MS}, \cite[Theorem~2.33]{Mordukhovich}, 
and the comments in Thibault~\cite[4.8]{thibault}.

\begin{lemma}[Fuzzy sum rule for Fr\'echet subdifferentials in Asplund spaces]\label{lem2.1}
Let $\mathbb{X}$ be an Asplund space and let 
$\varphi_i : \mathbb{X} \to \mathbb{R} \cup \{+\infty\}$, $i=1,2$, 
be proper lower semicontinuous functions, with at least one of them locally Lipschitz 
at some point $\bar{x} \in \mathrm{dom}(\varphi_1) \cap \mathrm{dom}(\varphi_2)$. 
Then, for any $\varepsilon > 0$, we have
\[
\widehat{\partial}(\varphi_1 + \varphi_2)(\bar{x}) 
\subseteq \bigcup \left\{
\widehat{\partial} \varphi_1(x_1) + \widehat{\partial} \varphi_2(x_2) :
\begin{array}{l}
x_i \in \mathbf{B}(\bar{x},\varepsilon),\\[0.1cm]
|\varphi_i(x_i) - \varphi(\bar{x})| < \varepsilon, \ i=1,2
\end{array}
\right\}
+ \varepsilon \, \mathbf{B}_{\mathbb{X}^*}.
\]
\end{lemma}

\medskip

\noindent
The next lemma, taken from Ngai and Th\'era~\cite[Proposition~2.1]{NT2001}, will also be used in the sequel.

\begin{lemma}\label{lem2.2a}
Let $\mathbb{X}$ be a Banach space, $\varphi : \mathbb{X} \to \mathbb{R} \cup \{+\infty\}$ 
be a proper lower semicontinuous function, and let $(\bar{x},\alpha) \in \mathrm{epi}(\varphi)$. 
Then, for any $\lambda \neq 0$, the following equivalence holds:
\[
(x^* , -\lambda) \in \widehat{N}(\mathrm{epi}(\varphi), (\bar{x},\alpha))
\ \Longleftrightarrow\ 
\lambda > 0, \quad \alpha = \varphi(\bar{x}), \quad \frac{x^*}{\lambda} \in \widehat{\partial} \varphi(\bar{x}).
\]
\end{lemma}

\medskip

\noindent
We also recall the following technical lemma (see \cite[Lemma~3.6]{NT2001} 
and \cite[Lemma~3.7]{AusDanThi05} for details), which will be useful in our analysis.

\begin{lemma}\label{lem2.2}
Let $\mathbb{X}$ be an Asplund space and $A \subseteq \mathbb{X}$ be a nonempty closed set. 
If $x \in \mathbb{X} \setminus A$ and $x^* \in {\widehat{\partial}{\bf d}(\cdot, A)(x)}$, then for any $\varepsilon > 0$ 
there exist $a \in A$ and $a^* \in \widehat{\mathbf{N}}(A, a)$ such that
\[
\|x-a\| < {\bf d}(x, A) + \varepsilon 
\quad \text{and} \quad 
\|x^* - a^*\| < \varepsilon.
\]
\end{lemma}

The following lemma, cited from  Zheng and Ng \cite[Theorem 3.1]{ZN}, is used in our analysis.

\begin{lemma}\label{lem2.3}
	Let $\mathbb{X}$  be  an Asplund space,  $A$ be a nonempty closed
	subset of $ \mathbb{X}$ and $x\not\in A$. Then for any $\beta\in (0,\;1)$
	there exist $z\in A$ and $z^*\in
	 \widehat{\bf N}(A, z)$ with $\|z^*\|=1$ such that
	$$\beta\|x-z\|<\min\{\mathbf{d}(x, A),\;\langle z^*,x-z\rangle\}.$$
\end{lemma}


We conclude this section with the following lemma, cited from \cite[Lemma 3.1]{WTY2024-SVVA}.  

\begin{lemma}\label{lem3.2}
Let $\mathbb{X}, \mathbb{Y}$ be Banach spaces, $C \subseteq \mathbb{Y}$ a closed convex set, and $\Psi:\mathbb{X}\to\mathbb{Y}$ a continuously differentiable mapping. Let $\bar{x} \in A := \Psi^{-1}(C)$ be such that $\nabla \Psi(\bar{x})$ is surjective. Then there exist constants $\ell, L, r \in (0, +\infty)$ such that
\begin{equation*}\label{3-24-polished}
\widehat{\mathbf{N}}(A,x) \cap \ell \mathbf{B}_{\mathbb{X}^*} 
\subseteq \nabla \Psi(x)^* \Big( \mathbf{N}(C, \Psi(x)) \cap \mathbf{B}_{\mathbb{Y}^*} \Big)
\subseteq \widehat{\mathbf{N}}(A,x) \cap L \mathbf{B}_{\mathbb{X}^*}, 
\quad \forall x \in \mathbf{B}(\bar{x}, r) \cap A.
\end{equation*}
\end{lemma}

\setcounter{equation}{0}

\section{Metric Subregularity and $\mathbf{BCQ}$ of Multifunctions in Asplund Spaces}

In this section, we investigate metric subregularity and the $\mathbf{BCQ}$ property in terms of limiting normal cones and coderivatives for closed multifunctions in Asplund spaces. We also establish an implication between these two properties, which in turn leads to characterizations of Asplund spaces.

We start by recalling the definition of metric subregularity for a multifunction. Recall that $F$ is said to be metrically subregular at $(\bar x,\bar y)\in {\rm gph}(F)$,  if there exist $\tau,r\in (0, +\infty)$ such that
\begin{equation}\label{3.1}
 {\bf d}(x, F^{-1}(\bar y))\leq\tau {\bf d}(\bar y, F(x)),\ \ \forall x\in \mathbf{B}(\bar x, r).
\end{equation}

Given two Banach spaces $\mathbb{X}$ and $\mathbb{Y}$, let $\Gamma(\mathbb{X},\mathbb{Y})$ denote the set of all multifunctions $F:\mathbb{X}\rightrightarrows\mathbb{Y}$ with closed graph; that is,  {${\rm gph}(F)$ is closed in $\mathbb{X}\times\mathbb{Y}$}.

%
\medskip

{When $F$ is convex, i.e., ${\rm gph}(F)$ is a convex subset of $\mathbb{X}\times \mathbb{Y}$}, Zheng and Ng \cite{ZN2007} introduced the following concept.

\begin{definition}  [Zheng and Ng \cite{ZN2007} $\mathbf{BCQ}$ for convex multifunctions]
$F$ is said to satisfy the \emph{$\mathbf{BCQ}$} at $(\bar{x},\bar{y})$ if
\begin{equation}\label{3.2a}
    \mathbf{N}\big(F^{-1}(\bar{y}),\bar{x}\big) = D^*F(\bar{x},\bar{y})(\mathbb{Y}^*).
\end{equation}

In the convex case, {noting that ${\rm gph}(F)$ is a convex subset}, the inclusion
\[
D^*F(\bar{x},\bar{y})(\mathbb{Y}^*) \subseteq \mathbf{N}\big(F^{-1}(\bar{y}),\bar{x}\big)
\]
always holds, and so \eqref{3.2a} is equivalent to
\[
\mathbf{N}\big(F^{-1}(\bar{y}),\bar{x}\big) \subseteq D^*F(\bar{x},\bar{y})(\mathbb{Y}^*).
\]
\end{definition}

Motivated by this observation, we extend the $\mathbf{BCQ}$ to the general (nonconvex) setting using limiting normal cones and coderivatives.

\begin{definition}[Limiting $\mathbf{BCQ}$] 
Let $F\in\Gamma(\mathbb{X},\mathbb{Y})$ and $(\bar{x},\bar{y})\in{\rm gph}(F)$.  
We say that $F$ satisfies the \emph{limiting $\mathbf{BCQ}$} at $(\bar{x},\bar{y})$ if
\begin{equation}\label{3.2}
    \mathbf{N}\big(F^{-1}(\bar{y}),\bar{x}\big) \subseteq D^*F(\bar{x},\bar{y})(\mathbb{Y}^*).
\end{equation}
\end{definition}

We will show that for closed multifunctions between Asplund spaces, metric subregularity implies the limiting $\mathbf{BCQ}$, and that this implication can be used to characterize Asplund spaces.  
Before proceeding, we present the following lemma, which is of independent interest.

\begin{lemma}\label{lem3.1}
Let $\mathbb{X}$ be a Banach space, $\mathbb{Y} := \mathbb{X}^m$ be equipped with the $\ell^1$-norm, and let $A_1,\dots,A_m$ be closed subsets of $\mathbb{X}$ with nonempty intersection.  
Define $F:\mathbb{X}\rightrightarrows\mathbb{Y}$ by
\[
F(x) := (x-A_1) \times \cdots \times (x-A_m), \quad \forall\, x\in\mathbb{X}.
\]
Let $(\bar{x},\bar{y})\in{\rm gph}(F)$ with $\bar{y} = (\bar{y}(1),\dots,\bar{y}(m))$. Then:
\begin{itemize}
    \item[\rm(i)] We have
    \[
    {\rm dom}(\widehat{D}^*F(\bar{x},\bar{y})) \subseteq \widehat{\mathbf{N}}(A_1,\bar{x}-\bar{y}(1)) \times \cdots \times \widehat{\mathbf{N}}(A_m,\bar{x}-\bar{y}(m)),
    \]
    and for all $y^* = (y^*(1),\dots,y^*(m))$ in ${\rm dom}(\widehat{D}^*F(\bar{x},\bar{y}))$,
    \[
    \widehat{D}^*F(\bar{x},\bar{y})(y^*) = \left\{ x^* \in \mathbb{X}^* : x^* = \sum_{i=1}^m y^*(i) \right\}.
    \]
    \item[\rm(ii)] We have
\begin{equation}\label{3.4}
   { {\rm dom}(D^*F(\bar{x},\bar{y})) \subseteq \mathbf{N}(A_1,\bar{x}-\bar{y}(1)) \times \cdots \times \mathbf{N}(A_m,\bar{x}-\bar{y}(m))},
\end{equation}
    and for all such  {$y^* = (y^*(1),\dots,y^*(m))$ in ${\rm dom}({D}^*F(\bar{x},\bar{y}))$},
\begin{equation}\label{3.3}
  {   D^*F(\bar{x},\bar{y})(y^*) \subseteq \left\{ x^* \in \mathbb{X}^* : x^* = \sum_{i=1}^m y^*(i) \right\}}.
\end{equation}
    Moreover, this inclusion becomes an equality when $\mathbb{X}$ is Asplund.
\end{itemize}
\end{lemma}

\begin{proof} It suffices to prove (ii). Let $y^*=(y^*(1),\cdots,y^*(m))\in {\rm dom}(D^*F(\bar x, \bar y))$. Take any $x^*\in D^*F(\bar x, \bar y)(y^*)$. Then 
$$
(x^*,-y^*)\in \mathbf{N}({\rm gph}(F),(\bar x,\bar y))
$$ 
and there exist $\epsilon_k\downarrow 0$, $(x_k,y_k)\xrightarrow{{\rm gph}(F)}(\bar x, \bar y)$ and $(x_k^*,y_k^*)\stackrel{w^*}\longrightarrow (x^*,y^*)$ such that
$$
(x_k^*,-y_k^*)\in\widehat{\mathbf{N}}_{\epsilon_k}({\rm gph}(F),(x_k,y_k)),\ \ \forall k.
$$
Thus, for any $k\in \mathbb{N}$, one has
\begin{equation}\label{3.5}
  \limsup_{(x,y)\xrightarrow{{\rm gph}(F)} (x_k,y_k)}\frac{\langle (x_k^*,-y_k^*), (x-x_k,y-y_k)\rangle}{\|x-x_k\|+\|y-y_k\|}\leq\epsilon_k.
\end{equation}
For \eqref{3.4}, %
it suffices to prove that $y^*(j)\in \mathbf{N}(A_j, \bar x-\bar y(j))$ for all $j$.

Let $j\in\{1,\cdots,m\}$ be fixed. Note that $(x_k,y_k)\in{\rm gph}(F)$ and thus $x_k-y_k(j)\in A_j$. Then for any $h \rightarrow x_k-y_k(j)$ with $h\in A_j$, take $$
x:=x_k, y:=(y_k(1),\cdots,x_k-h,\cdots,y_k(m)).
$$
Then $(x,y)\xrightarrow{{\rm gph}(F)} (x_k,y_k)$ as $h \xrightarrow{A_j} x_k-y_k(j)$ and so \eqref{3.5} gives that
$$
\limsup_{h\stackrel{A_j}\longrightarrow x_k-y_k(j)}\frac{\langle y_k^*(j), h-(x_k-y_k(j))\rangle}{\|h-(x_k-y_k(j))\|}\leq\epsilon_k.
$$
This implies that $y_k^*(j)\in\widehat {\bf N}_{\epsilon_k}(A_j, x_k-y_k(j))$ and consequently $y^*(j)\in \mathbf{N}(A_j, \bar x-\bar y(j))$ (thanks to $x_k-y_k(j)
\stackrel{A_j}\longrightarrow \bar x-\bar y(j)$ and $(y_k^*(j))\stackrel{w^*}\longrightarrow y^*(j)$).

We next prove \eqref{3.3}. Let $b\in {\bf B}_\mathbb{X}$. For any $t>0$ sufficiently small, take
$$
x:=x_k+tb, y:=(y_k(1)+tb,\cdots,y_k(m)+tb).
$$
Then one can verify that $(x,y)\in {\rm gph}(F)$ and $(x,y)\rightarrow (x_k,y_k)$ as $t\rightarrow 0^+$. This and \eqref{3.5} imply that 
$$
\langle x_k^*-\sum_{i=1}^my_k^*(i), b\rangle\leq (m+1)\epsilon_k.
$$
Taking limits as $k\rightarrow\infty$, one has
$$
\langle x^*-\sum_{i=1}^my^*(i), b\rangle\leq 0. %
$$
This means that $x^*=\sum_{i=1}^my^*(i)$ due to the arbitrariness of $b\in {\bf B}_\mathbb{X}$ and so \eqref{3.3} holds. 

Suppose that $\mathbb{X}$  is Asplund. It remains to prove ``$\supseteq$" in \eqref{3.3}.

Let $x^*:=\sum_{i=1}^my^*(i)$. Let $i\in \{1,\cdots, m\}$ be fixed. Then $y^*(i)\in \mathbf{N}(A_i, \bar x-\bar y(i))$ and thus there exist $u_k(i)\stackrel{A_i}\longrightarrow\bar x-\bar y(i)$ and $y^*_k(i)\stackrel{w^*}\longrightarrow y^*(i)$ such that $y^*_k(i)\in\widehat {\bf N}(A_i, \bar x-\bar y(i))$ for all $k$. For any $k$, let $v_k:=(\bar x-u_k(1), \cdots, \bar x-u_k(m))$. Then one can verify that $v_k\rightarrow \bar y$ and $v_k\in F(\bar x)$ due to $u_k(i)\in A_i$. Let $y_k^*:=(y^*_k(1),\cdots,y^*_k(m))$ and $x_k^*:=\sum_{i=1}^my^*(i)$. Then $(x_k^*,y_k^*)\stackrel{w^*}\rightarrow (x^*,y^*)$. 
Note that for any $(x,y)\xrightarrow{{\rm gph}(F)} (\bar x,v_k)$, one has
\begin{eqnarray*}
	\langle x_k^*, x-\bar x\rangle +\langle-y_k^*, y-v_k\rangle &=& 	\Big\langle \sum_{i=1}^my_k^*(i), x-\bar x\Big\rangle + \sum_{i=1}^m\langle-y_k^*(i), y(i)-(\bar x-u_k(i))\rangle\\
	&=& \sum_{i=1}^m\langle y_k^*(i), x-\bar x-y(i)+(\bar x-u_k(i))\rangle\\
	&=& \sum_{i=1}^m\langle y_k^*(i), x-y(i)-u_k(i)\rangle
\end{eqnarray*}
and it follows from $y_k^*(i)\in\widehat{\mathbf{N}}(A_i, \bar x-\bar y(i))$ that
\begin{equation*}
	\limsup_{(x,y)\xrightarrow{{\rm gph}(F)} (\bar x, v_k)}\frac{\langle (x^*,-y^*), (x-\bar x, y-v_k)\rangle}{\|x-\bar x\|+\|y-v_k\|}\leq 0.
\end{equation*}
This means that $(x_k^*, -y_k^*)\in\widehat{\mathbf{N}}({\rm gph}(F), (\bar x, v_k))$ and thus $(x^*, -y^*)\in \mathbf{N}({\rm gph}(F), (\bar x, \bar y))$ thanks to $v_k\in F(\bar x)$ and $(x_k^*,y_k^*)\stackrel{w^*}\rightarrow (x^*,y^*)$.  
 The proof is complete.
 \end{proof}

\medskip

The following theorem establishes characterizations of Asplund spaces in terms of metric subregularity and limiting $\mathbf{BCQ} $  of multifunctions.


\begin{theorem}\label{th3.1}
Let $\mathbb{X}$ be a Banach space. Then the following statements are equivalent:
\begin{itemize}
    \item[\rm (i)] $\mathbb{X}$ is an Asplund space.
    \item[\rm (ii)] For every Asplund space $\mathbb{Y}$ and every $F \in \Gamma(\mathbb{X}, \mathbb{Y})$  {being} metrically subregular at $(\bar x, \bar y) \in {\rm gph}(F)$, there exists $\delta \in (0, +\infty)$ such that $F$ satisfies the limiting $\mathbf{BCQ}$ at all $(x, \bar y)$ with $x \in \mathbf{B}(\bar x, \delta) \cap F^{-1}(\bar y)$.
    \item[\rm (iii)] For every Asplund space $\mathbb{Y}$ and every $F \in \Gamma(\mathbb{X}, \mathbb{Y})$  {being} metrically subregular at $(\bar x, \bar y) \in {\rm gph}(F)$, $F$ satisfies the limiting $\mathbf{BCQ}$ at $(\bar x, \bar y)$.
\end{itemize}
\end{theorem}

\begin{proof}  (i)$\,\Rightarrow\,$(ii): Suppose that $\mathbb{X}$ is an Asplund space. Let $\mathbb{Y}$ be an Asplund space and $F\in\Gamma(\mathbb{X}, \mathbb{Y})$ be metrically subregular at $(\bar x,\bar y)\in{\rm gph}(F)$. Then there exist $\tau,r  \in  (0,+\infty)$ such that \eqref{3.1} holds.

For any $(x, y)\in \mathbb{X}\times
\mathbb{Y}$, let $\|(x, y)\|_{\tau}:=\frac{\tau+1}{\tau}\|x\|+\|y\|$. Then  $(\mathbb{X}\times
\mathbb{Y}, \|\cdot\|_{\tau})$ is an Asplund space and the unit ball of its
dual space is $(\frac{\tau+1}{\tau}\mathbf{B}_{\mathbb{X}^*})\times \mathbf{B}_{\mathbb{Y}^*}$. Let $\delta:=\frac{r}{2}$. We claim that
\begin{equation}\label{3.6}
\mathbf{d}(x, F^{-1}(\bar y))\leq \tau(\mathbf{d}_{\|\cdot\|_{\tau}}((x, y), {\rm gph}(F))+\|y-\bar y\|),\ \ \forall (x, y)\in  \mathbf{B}(\bar x, \delta)\times \mathbb{Y},
\end{equation}
where $\mathbf{d}_{\|\cdot\|_{\tau}}((x, y), {\rm gph}(F)):=\inf\{\|(x,y)-(u,v)\|_{\tau}: (u,v)\in{\rm gph}(F)\}$.

Suppose   {on} the contrary that there exists $(x_0, y_0)\in  \mathbf{B}(\bar x, \delta)\times \mathbb{Y}$ such
that
$$
\mathbf{d}(x_0,  F^{-1}(\bar y)) > \tau(\mathbf{d}_{\|\cdot\|_{\tau}}((x_0, y_0),
{\rm gph}(F))+\|y_0-\bar y\|).
$$
This implies that there exists $u\in \mathbb{X}$ such that
$$
\mathbf{d}(x_0, F^{-1}(\bar y))> \tau(\frac{\tau+1}{\tau}\|u-x_0\|+\mathbf{d}(y_0, F(u))+\|y_0-\bar y\|.
$$
Thus,
$$
\mathbf{d}(x_0, F^{-1}(\bar y))>\|u-x_0\|+\tau\mathbf{d}(\bar y, F(u)).
$$
Since
$$
\|u-\bar x\|\leq\|u-x_0\|+\|x_0-\bar x\|<\mathbf{d}(x_0, F^{-1}(\bar y))+\|x_0-\bar x\|\leq
2\|x_0-\bar x\|<r,
$$
then \eqref{3.1} gives that
$$
\mathbf{d}(x_0, F^{-1}(\bar y))>\|u-x_0\|+\mathbf{d}(u, F^{-1}(\bar y))\geq \mathbf{d}(x_0, F^{-1}(\bar y)),
$$
which is a contradiction. Hence \eqref{3.6} holds.

Let $x\in  \mathbf{B}(\bar x, \delta)\cap F^{-1}(\bar y)$ and $x^*\in \mathbf{N}(F^{-1}(\bar y),
x)$. Then there exists a sequence $\{(x_n,x_n^*)\}$ in $F^{-1}(\bar y)\times \mathbb{X}^*$ such that 
$$
x_n\rightarrow x, \ x_n^*\stackrel{w^*}\rightarrow x^*\ \ {\rm and} \ \ x_n^*\in\widehat{\mathbf{N}}(F^{-1}(\bar y),x_n),\ \forall n.
$$
Applying  the   Banach-Steinhaus theorem, there is $M>0$ such that $\|x_n^*\|\leq M$ for all $n$. For any $n\in\mathbb{N}$, let $\widetilde{x_n^*}:=\frac{x_n^*}{M}$ and then  \cite[Corollary 1.96]{Mordukhovich} gives that
$$\widetilde{x_n^*}\in\widehat{\mathbf{N}}(F^{-1}(\bar y),x_n)\cap \mathbf{B}_{\mathbb{X}^*}=\widehat{\partial}\mathbf{d}(\cdot, F^{-1}(\bar y))(x_n).$$ 
We claim that
\begin{equation}\label{3.8-250618}
	(\frac{\widetilde{x_n^*}}{\tau},0)\in\widehat\partial(\mathbf{d}_{\|\cdot\|_{\tau}}(\cdot,
	{\rm gph}(F))+\phi)(x_n,\bar y)
\end{equation}
where $\phi(u,v):=\|v-\bar y\|$ for any $(u,v)\in\mathbb{X}\times\mathbb{Y}$.

Indeed, for any
$\varepsilon>0$ there exits $\delta_n\in (0, \delta-\|x-\bar x\|)$ such that
$$
\langle \widetilde{x_n^*}, u-x_n\rangle\leq \mathbf{d}(x, F^{-1}(\bar y))+\tau\varepsilon\|u-x_n\|,\ \ \forall u\in  \mathbf{B}(x_n, \delta_n).
$$
This and \eqref{3.6} imply that for any $(u,v)\in  \mathbf{B}(x_n, \delta_n)\times \mathbf{B}(\bar y, \delta_n)$, one has
\begin{eqnarray*}
\langle \widetilde{x_n^*}, u-x_n\rangle\leq\tau (\mathbf{d}_{\|\cdot\|_{\tau}}((u, v),
{\rm gph}(F))+\|v-\bar y\|)+\tau \varepsilon\|u-x_n\|
\end{eqnarray*}
and thus
$$
\langle \frac{\widetilde{x_n^*}}{\tau}, u-x_n\rangle\leq \mathbf{d}_{\|\cdot\|_{\tau}}((u, v),
{\rm gph}(F))+\phi(u,v)+\varepsilon\|u-x_n\|.
$$
This implies that \eqref{3.8-250618} holds.

Since $\mathbb{X}\times\mathbb{Y}$ is  an  Asplund space, by virtue of \cref{lem2.1}, there exist $(u_n,v_n), (w_n,z_n)\in  \mathbf{B}(x_n,\frac{1}{n})\times  \mathbf{B}(\bar y,\frac{1}{n})$ such that
\begin{eqnarray*}
(\frac{\widetilde{x_n^*}}{\tau},0)&\in& \widehat\partial\mathbf{d}_{\|\cdot\|_{\tau}}(\cdot,
{\rm gph}(F))(u_n,v_n)+\widehat\partial\phi(w_n,z_n)+\frac{1}{n}(\mathbf{B}_{\mathbb{X}^*}\times \mathbf{B}_{\mathbb{Y}^*})\\
&\subseteq& \widehat\partial\mathbf{d}_{\|\cdot\|_{\tau}}(\cdot,
{\rm gph}(F))(u_n,v_n)+\{0\}\times \mathbf{B}_{\mathbb{Y}^*}+\frac{1}{n}(\mathbf{B}_{\mathbb{X}^*}\times \mathbf{B}_{\mathbb{Y}^*}).
\end{eqnarray*}
Then there {exists}  $(u_n^*,v_n^*)\in\widehat\partial\mathbf{d}_{\|\cdot\|_{\tau}}(\cdot,
{\rm gph}(F))(u_n,v_n)$ and $b_n^*\in\mathbf{B}_{\mathbb{Y}^*}$ such that
\begin{equation}\label{3.7}
 (\frac{\widetilde{x_n^*}}{\tau},0) \in  {(u_n^*,v_n^*)}+(0,b_n^*)+\frac{1}{n}(\mathbf{B}_{\mathbb{X}^*}\times \mathbf{B}_{\mathbb{Y}^*}).
\end{equation}
For $(u_n^*,v_n^*)\in\widehat\partial\mathbf{d}_{\|\cdot\|_{\tau}}(\cdot,
{\rm gph}(F))(u_n,v_n)$, \cref{lem2.2} gives the existence of   $(\widetilde{u_n}, \widetilde{v_n})\in{\rm gph}(F)$ and $(\widetilde{u_n^*}, \widetilde{v_n^*})\in {\widehat{\mathbf{N}}({\rm gph}(F),(\widetilde{u_n}, \widetilde{v_n}))}$ such that
\begin{equation}\label{3.8}
\|(\widetilde{u_n}, \widetilde{v_n})-(u_n,v_n)\|<\mathbf{d}_{\|\cdot\|_{\tau}}(({u_n}, {v_n}),
{\rm gph}(F))+\frac{1}{n}\ \ {\rm and} \ \ \|(\widetilde{u_n^*}, \widetilde{v_n^*})-(u_n^*,v_n^*)\|<\frac{1}{n}.
\end{equation}
This and \eqref{3.7} imply that
\begin{equation}\label{3.9}
   (\frac{\widetilde{x_n^*}}{\tau},0) {\in(\widetilde{u_n^*},\widetilde{v_n^*})}+(0,b_n^*)+\frac{2}{n}(\mathbf{B}_{\mathbb{X}^*}\times \mathbf{B}_{\mathbb{Y}^*}).
\end{equation}
Note that 
$$
(u_n^*,v_n^*)\in\widehat\partial\mathbf{d}_{\|\cdot\|_{\tau}}(\cdot, {\rm gph}(F))(u_n,v_n)\subseteq \frac{\tau+1}{\tau}\mathbf{B}_{\mathbb{X}^*}\times\mathbf{B}_{\mathbb{Y}^*}
$$
and $\mathbf{B}_{\mathbb{X}^*},\mathbf{B}_{\mathbb{Y}^*}$ are weak$^*$-sequentially compact (as $\mathbb{X},\mathbb{Y}$ are Asplund spaces), and thus without loss of  {generality} we can assume that
$$
(u_n^*,v_n^*)\stackrel{w^*}\rightarrow (u^*,v^*) \ \ {\rm and}\ \  b_n^*\stackrel{w^*}\rightarrow b^*\in \mathbf{B}_{\mathbb{Y}^*}
$$
(taking subsequence if necessary). This and \eqref{3.9} imply that $(u^*,v^*)\in  {\mathbf{N}({\rm gph}(F),(x,\bar y))}$ (thanks to $(u_n,v_n)\rightarrow (x,\bar y)$). By taking the   limit (with respect to the weak$^*$-topology) as $n\rightarrow \infty$ in \eqref{3.9}, one has
$$
(\frac{x^*}{M\tau},0)=(u^*, v^*)+(0,b^*), 
$$
and consequently
$$
x^*\in D^*F(x,\bar y)(\tau M b^*)\subseteq D^*F(x,\bar y)(\mathbb{Y}^*).
$$
This means that $F$ satisfies the limiting $\mathbf{BCQ} $  at $ {(x,\bar y)}$.

Note that (ii)$\; \Rightarrow \; $(iii) follows immediately, and it suffices to prove that (iii)$\;\Rightarrow\;$(i).

Suppose  {on}  the contrary that $\mathbb{X}$ is not Asplund. Following \cite{MW2000,FM1998}, we can represent $\mathbb{X}$ in the form $\mathbb{X}=\mathbb{Z}\times \mathbb{R}$ with the norm $\|(z,\alpha)\|:=\|z\|+|\alpha|$ for any $x=(z,\alpha)\in \mathbb{X}$. Then $\mathbb{ Z}$ is not an Asplund space. By virtue of \cite[Theorem 1.5.3]{DGZ1993} (also \cite[Theorem 2.1]{FM1998}), there exists an equivalent norm $|||\cdot|||$ on $\mathbb{Y}$ and $\gamma>0$ such that 
\begin{equation}\label{3.10}
\frac{1}{2}\|z\|\leq|||z|||\leq \|z\|\ \ {\rm and} \ \  \limsup_{h\rightarrow 0}\frac{|||z+h|||+|||z-h|||-2|||z|||}{\|h\|}>\gamma, \ \forall z\in \mathbb{Z}.
\end{equation}
Let $\varphi:\mathbb{Z}\rightarrow \mathbb{R}$ be defined as $\varphi(z):=-|||z|||,\forall z\in \mathbb{Z}$ and 
$$
A_1:=\{0_\mathbb{Z}\}\times (-\infty, 0], A_2:=\epi(\varphi)\ \ {\rm and} \ \ \bar x:=0_\mathbb{X}.
$$
Let $\mathbb{Y}:=\mathbb{X}^2$ be equipped with the   $\ell^1$-norm. Then $\mathbb{Y}$ is not the Asplund space. Define $F:\mathbb{X}\rightarrow\mathbb{Y}$ as follows:
$$
F(x):=(x-A_1)\times(x-A_2),\ \ \forall x\in\mathbb{X}.
$$
Denote $\bar y:=(0_\mathbb{X},0_\mathbb{X})$. Then $(\bar x,\bar y)\in{\rm gph}(F)$ and $F^{-1}(\bar y)=A_1\cap A_2=\{\bar x\}$. We first show that 
\begin{equation}\label{3.11}
  \mathbf{d}(x, F^{-1}(\bar y))\leq 2(\mathbf{d}(\bar y, F(x)))
\end{equation}
holds for all $x=(z,\alpha)\in \mathbb{Z}\times \mathbb{R}$, which means that $F$ is metrically subregular at $(\bar x,\bar y)$.

Indeed, for any $x=(z,\alpha)\in {\mathbb{Z} \times \mathbb{R}}$, one has
\begin{equation}\label{3.12}
  \mathbf{d}(x, F^{-1}(\bar y))=\|z\|+|\alpha|\ \ {\rm and} \ \ \mathbf{d}(\bar y, F(x))=\mathbf{d}(x, A_1)+\mathbf{d}(x, A_2).
\end{equation}

If $\alpha\geq 0$, then \eqref{3.11} holds since $\mathbf{d}(x, F^{-1}(\bar y))=\mathbf{d}(\bar y, F(x)))$.

If $-|||z|||\leq \alpha <0$, then \eqref{3.10} gives that
$$
\mathbf{d}(x,  F^{-1}(\bar y))=\|z\|+|\alpha|\leq 2\|z\|= 2\mathbf{d}(\bar y, F(x))
$$
and thus \eqref{3.11} holds.

If $\alpha <-|||z|||$, then for any $(u,\lambda)\in A_2$, one has $-|||u|||\leq \lambda$ and 
\begin{eqnarray*}
-|||z|||-\alpha&\leq& -|||z|||-\alpha+|||u|||+\lambda\\
&\leq&\big||||z|||-|||u|||\big|+|\alpha-\lambda|\\
&\leq& |||z-u|||+|\alpha-\lambda|\\
&\leq& \|z-u\|+|\alpha-\lambda|,
\end{eqnarray*}
where the last inequality  follows from  \eqref{3.10}. This implies that
$$
-|||z|||-\alpha\leq \mathbf{d}(x, A_2)
$$
and it follows from \eqref{3.10} and \eqref{3.12} that
$$
\mathbf{d}(x, F^{-1}(\bar y))=\|z\|-\alpha=\|z\|+|||z|||-|||z|||-\alpha\leq 2\|z\|+\mathbf{d}(x, A_2)\leq 2\mathbf{d}(\bar y,F(x))
$$
(thanks to $\mathbf{d}(x, A_1)=\|z\|$). Hence \eqref{3.11} holds.

We next prove that 
\begin{equation}\label{3.15-250708}
	  {\bf N}(A_2,\bar x)=\{(0,0)\}.
\end{equation}
Let $x^*=(z^*,\lambda)\in  {\bf N}(A_2,\bar x)$. Then there exist $(z_k,\alpha_k)\stackrel{A_2}\longrightarrow (\bar z,\bar\alpha)$ and $(z_k^*,\lambda_k)\stackrel{w^*}\longrightarrow (z^*,\lambda)$ such that $(z_k^*,\lambda_k)\in  {\widehat{\bf N}(A_2,(z_k,\alpha_k))}$ for all $k$. Then $ {\lambda_k}\leq 0$. We claim that $ {\lambda_k}=0$. (Indeed, if not, then \cref{lem2.2} implies that $\lambda_k<0$ and $\frac{z_k^\ast }{-\lambda_k}\in\widehat\partial \varphi(z_k)$. Thus,
$$
\liminf_{\|h\|\rightarrow 0}\frac{-|||z_k+h|||+|||z_k|||-\langle\frac{z_k^\ast }{\lambda_k}, h\rangle}{\|h\|}\geq 0.
$$
Hence
$$
\limsup_{\|h\|\rightarrow 0}\frac{|||z_k+h|||+|||z_k-h|||-2|||z_k|||}{\|h\|}\leq 0,
$$
which is a contradiction with \eqref{3.10}). This implies that $(z_k,0)\in {\widehat{\bf N}(A_2,(z_k,\alpha_k))}$. Then for any $b\in{\bf B}_{\mathbb{X}}$, one has
$$
\limsup_{t\rightarrow 0^+}\frac{\langle z_k^*, tb\rangle}{t\|b\|+\big|-|||z_k+tb|||+|||z_k|||\big|}\leq 0
$$
and consequently $\langle z_k^*, b\rangle\leq 0$. This means that $z_k^*=0$ and \eqref{3.15-250708} holds.

Since $F$ satisfies the limiting $\mathbf{BCQ} $  at $(\bar x,\bar y)$, it follows from   \cref{lem3.1}  that
\begin{equation}\label{3.13-a}
  D^*F(\bar x,\bar y)(\mathbb{Y}^*)\subseteq \mathbf{N}(A_1,\bar x)+\mathbf{N}(A_2,\bar x).
\end{equation}
Note that %
\begin{equation}\label{3.17-250710}
\mathbf{N}(F^{-1}(\bar y),\bar x)= {\mathbb{Z}^*}\times \mathbb{R},\ \ \mathbf{N}(A_1,\bar x)= {\mathbb{Z}^*}\times [0,+\infty) \ \  {\rm and} \ \ \mathbf{N}(A_2,\bar x)=\{(0,0)\}. 
\end{equation}
This and \eqref{3.13-a} imply that $F$ does not satisfy the limiting $\mathbf{BCQ} $  at $(\bar x,\bar y)$ since 
\begin{equation}\label{3.18-250710}
\mathbf{N}(F^{-1}(\bar y),\bar x)\not\subseteq \mathbf{N}(A_1,\bar x) + \mathbf{N}(A_2,\bar x),
\end{equation}
which contradicts (iii). The proof is complete.
\end{proof}
\medskip

Further, we show that metric subregularity may lead to ``fuzzy'' inclusions expressed in terms of Fr\'echet coderivatives and normal cones. 
Moreover, such inclusions can be employed to obtain characterizations of Asplund spaces.

\begin{theorem}\label{theorem3.2}
	Let $\mathbb{X}$ be a Banach space. Then the following statements are equivalent:
	\begin{itemize}
		\item [\rm(i)] $\mathbb{X}$ is an Asplund space.
		\item [\rm(ii)] For every Asplund space $\mathbb{Y}$ and every $F\in\Gamma(\mathbb{X}, \mathbb{Y})$  {being} metrically subregular at $(\bar x,\bar y)\in{\rm gph}(F)$, there exists $\delta\in (0,+\infty)$ such that for all  $\varepsilon>0$,  {and $x\in \mathbf{B}(\bar x,\delta)\cap F^{-1}(\bar y) $ one has}
		 \begin{equation}\label{3-15a}
			\widehat{\mathbf{N}}(F^{-1}(\bar y), x)\subseteq \bigcup\left\{\widehat D^*F(u,v)(\mathbb{Y}^*):(u,v)\in \mathbf{B}((x,\bar y),\epsilon)\cap {\rm gph}(F)\right\}+\epsilon\mathbf{B}_{\mathbb{X}^*}.
		\end{equation}
		\item [\rm(iii)] For every Asplund space $\mathbb{Y}$ and every $F\in\Gamma(\mathbb{X}, \mathbb{Y})$  {being} metrically subregular at $(\bar x,\bar y)\in{\rm gph}(F)$, then 
		 for all $\varepsilon>0$, one has 
		\begin{equation}\label{3-16a}
			\widehat{\mathbf{N}}(F^{-1}(\bar y), \bar x)\subseteq \bigcup\left\{\widehat D^*F(u,v)(\mathbb{Y}^*):(u,v)\in \mathbf{B}((\bar x,\bar y),\epsilon)\cap {\rm gph}(F)\right\}+\varepsilon\mathbf{B}_{\mathbb{X}^*}.
		\end{equation}
	\end{itemize}
\end{theorem}

\begin{proof} (i)$\,\Rightarrow\,$(ii): Suppose that $\mathbb{X}$ is an Asplund space. Let $\mathbb{Y}$  {be an} Asplund space and $F\in\Gamma(\mathbb{X}, \mathbb{Y})$ be metrically subregular at $(\bar x,\bar y)\in{\rm gph}(F)$. Then there exist $\tau,r\in(0,+\infty)$ such that \eqref{3.1} holds. By 
defining $\|(x, y)\|_{\tau}:=\frac{\tau+1}{\tau}\|x\|+\|y\|$ for any $(x, y)\in \mathbb{X}\times
\mathbb{Y}$, one has that \eqref{3.6} holds with $\delta:=\frac{r}{2}$.

Let $\varepsilon>0$. Take any $x\in  \mathbf{B}(\bar x,\delta)\cap F^{-1}(\bar y)$ and $x^*\in\widehat{\mathbf{N}}(F^{-1}(\bar y), x)\backslash \{0\}$. Then  $$\frac{x^*}{\|x^*\|}\in\widehat {\bf N}(F^{-1}(\bar y), x)\cap  \mathbf{B}_{\mathbb{X}^*}=\widehat \partial \mathbf{d}(\cdot,F^{-1}(\bar y))(x).$$ Using the proof of \eqref{3.8-250618}, one can verify that
$$(\frac{x^*}{\tau\|x^*\|},0)\in\widehat\partial(\mathbf{d}_{\|\cdot\|_{\tau}}(\cdot, {\rm gph}(F))+\phi)(x,\bar y).$$
where $\phi(u,v):=\|v-\bar y\|$ for any $(u,v)\in\mathbb{X}\times\mathbb{Y}$. 

Choose any $\epsilon_1\in (0,\epsilon)$ such that  $ { 2\tau\|x^*\|\epsilon_1<\epsilon}.$
By virtue of \cref{lem2.1}, there are  {two vectors $(x_1,y_1), (x_2,y_2)$ in  $\mathbf{B}(x,\epsilon_1)\times  \mathbf{B}(\bar y,\epsilon_1)$} such that
\begin{eqnarray*}
(\frac{x^*}{\tau\|x^*\|},0)&\in& \widehat\partial\mathbf{d}_{\|\cdot\|_{\tau}}(\cdot,
{\rm gph}(F))(x_1,y_1)+\widehat\partial\phi(x_2,y_2)+\epsilon_1(\mathbf{B}_{\mathbb{X}^*}\times \mathbf{B}_{\mathbb{Y}^*})\\
&\subseteq& \widehat\partial\mathbf{d}_{\|\cdot\|_{\tau}}(\cdot,
{\rm gph}(F))(x_1,y_1)+\{0\}\times \mathbf{B}_{\mathbb{Y}^*}+\epsilon_1(\mathbf{B}_{\mathbb{X}^*}\times \mathbf{B}_{\mathbb{Y}^*}).
\end{eqnarray*}
This implies that there exist $(x_1^*,y_1^*)\in\widehat\partial\mathbf{d}_{\|\cdot\|_{\tau}}(\cdot,
{\rm gph}(F))(x_1,y_1)$ and $b_1^*\in\mathbf{B}_{\mathbb{Y}^*}$ such that
\begin{equation}\label{3.16b}
(\frac{x^*}{\tau\|x^*\|},0)\in(x_1^*,y_1^*)+(0,b_1^*)+\epsilon_1(\mathbf{B}_{\mathbb{X}^*}\times \mathbf{B}_{\mathbb{Y}^*}).
\end{equation}
Note that $(x_1^*,y_1^*)\in\widehat\partial\mathbf{d}_{\|\cdot\|_{\tau}}(\cdot,
{\rm gph}(F))(x_1,y_1)$ and it follows from \cref{lem2.2} that there exist $(u_1,v_1)\in{\rm gph}(F)$ and $(u_1^*,v_1^*)\in {\widehat {\bf N}}({\rm gph}(F),(u_1,v_1))$ such that
\begin{equation}\label{3.17b}
\|(u_1,v_1)-(x_1,y_1)\|<\mathbf{d}_{\|\cdot\|_{\tau}}((u_1,v_1),
{\rm gph}(F))+\epsilon_1\ \ {\rm and} \ \ \|(u_1^*,v_1^*)-(x_1^*,y_1^*)\|<\epsilon_1.
\end{equation}
Then $(u_1,v_1)\in \mathbf{B}((x,\bar y),\epsilon)$ and \eqref{3.16b} gives that
\begin{equation*}\label{3.18b}
 (\frac{x^*}{\tau\|x^*\|},0)\in(u_1^*,v_1^*)+(0,b_1^*)+2\epsilon_1(\mathbf{B}_{\mathbb{X}^*}\times \mathbf{B}_{\mathbb{Y}^*}).
\end{equation*}
This and \eqref{3.17b} imply
$$
x^*\in\widehat D^*F(u_1,v_1)(\tau\|x^*\|\|v_1^*\|\mathbf{B}_{\mathbb{Y}^*})+2\tau\|x^*\|\epsilon_1\mathbf{B}_{\mathbb{X}^*}\subseteq \widehat D^*F(u_1,v_1)(\mathbb{Y}^*)+\varepsilon\mathbf{B}_{\mathbb{X}^*}
$$
(thanks to the choice of $\epsilon_1$). Hence \eqref{3-15a} holds. 

Note that (ii)$\,\Rightarrow\,$(iii) follows immediately, and thus we next prove that (iii)$\,\Rightarrow\,$(i).

Suppose  {on}  the contrary that $\mathbb{X}$ is not an Asplund space. Following \cite{MW2000,FM1998}, we can represent $\mathbb{X}$ in the form $\mathbb{X}=\mathbb{Z}\times \mathbb{R}$ with the norm $\|(z,\alpha)\|:=\|z\|+|\alpha|$ for any $x=(z,\alpha)\in \mathbb{X}$. Then $\mathbb{Z}$ is not an Asplund space. By virtue of \cite[Theorem 1.5.3]{DGZ1993} (also \cite[Theorem 2.1]{FM1998}), there exists an equivalent norm $|||\cdot|||$ on $\mathbb{Y}$ and $\gamma>0$ such that \eqref{3.10} holds. 
Let $\varphi: {\mathbb{Z}}\rightarrow \mathbb{R}$ be defined as $\varphi(z):=-|||z|||,\forall z\in \mathbb{Z}$ and 
$$
A_1:=\{0_{ {\mathbb{Z}}}\}\times (-\infty, 0],  A_2:={\rm epi}(\varphi)\ \ {\rm and} \ \ \bar x:=(0_{ {\mathbb{Z}}}, 0).
$$
Let $\mathbb{Y}:=\mathbb{X}^2$ be equipped with the   $\ell^1$-norm. Then $\mathbb{Y}$ is not the Asplund space. Define $F:\mathbb{X}\rightarrow\mathbb{Y}$ as follows:
$$
F(x):=(x-A_1)\times(x-A_2),\ \ \forall x\in\mathbb{X},
$$
and denote $\bar y:=(0_\mathbb{X},0_\mathbb{X})$. 

To complete the proof, we show that for any $\varepsilon>0$, one has
\begin{equation}\label{1}
	\widehat{\mathbf{N}}(F^{-1}(\bar y), \bar x)\not\subseteq\left\{\widehat D^*F(u,v)(\mathbb{Y}^*):(u,v)\in \mathbf{B}((\bar x,\bar y),\varepsilon)\cap {\rm gph}(F)\right\}+\varepsilon\mathbf{B}_{\mathbb{X}^*},
\end{equation}
which contradicts \eqref{3-16a}.

Let $\varepsilon>0$. Take any $(u,v)\in \mathbf{B}((\bar x,\bar y),\epsilon)\cap {\rm gph}(F)$ with $v=(v(1),v(2))$. Then $$u-v(1)=(z_1,\alpha_1)\in A_1\ \ {\rm and} \ \ u-v(2)=(z_2,\alpha_2)\in A_2.$$
This implies that $z_1=0_{\mathbb{Z}},\alpha_1\leq 0$ and thus
\begin{equation}\label{2}
	\widehat{\mathbf{N}}(A_1,u-v(1))=\left\{
	\begin{aligned}
		\mathbb{Z}^\ast \times\{0\},\ \ \ \ \ \ & \alpha_1<0, \\
		\mathbb{Z}^\ast \times [0,+\infty), \ & \alpha_1=0.
	\end{aligned}
	\right.
\end{equation}
We next prove that
\begin{equation}\label{3}
	\widehat{\mathbf{N}}(A_2,u-v(2))=\{(0,0)\}.
\end{equation}
Let $(z^\ast , -\lambda)\in\widehat{\mathbf{N}}(A_2,u-v(2))$. We claim that $\lambda=0$. 

Indeed, suppose on the contrary that $\lambda\not=0$. Applying \cref{lem2.2a}, we have $\lambda>0$ and $\frac{z^\ast }{\lambda}\in\widehat\partial \varphi(z_2)$. Thus,
$$
\liminf_{\|h\|\rightarrow 0}\frac{-|||z_2+h|||+|||z_2|||-\langle\frac{z^\ast }{\lambda}, h\rangle}{\|h\|}\geq 0.
$$
Hence
$$
\limsup_{\|h\|\rightarrow 0}\frac{|||z_2+h|||+|||z_2-h|||-2|||z_2|||}{\|h\|}\leq 0,
$$
which is a contradiction with \eqref{3.10}. Thus the claim follows.

For $(z^\ast , 0)\in\widehat{\mathbf{N}}(A_2,u-v(2))$, for any $\epsilon_0>0$, there exists $\delta_0>0$ such that
$$
\langle z^\ast , z-z_2\rangle\leq\epsilon_0(\|z-z_2\|+|\alpha-\alpha_2|),\ \ \forall (z,\alpha)\in \big((z_2,\alpha_2)+\delta_0 \mathbf{B}_{Z\times\mathbb{R}}\big)\cap A_2.
$$
For each $z\in z_2+\frac{\delta_0}{2} \mathbf{B}_\mathbb{Y}$, take $\alpha:=-|||z|||+|||z_2|||+\alpha_2$. Then one has 
$$
\|(z,\alpha)-(z_2,\alpha_2)\|=\|z-z_2\|+|\alpha-\alpha_2|\leq 2\|z-z_2\|<\delta_0
$$
and it follows from \eqref{3.10} that 
\begin{eqnarray*}
	\langle z^\ast , z-z_2\rangle\leq\epsilon_0(\|z-z_2\|+|\alpha-\alpha_2|)\leq 2\epsilon_0\|z-z_2\|.
\end{eqnarray*}
This implies that $z^\ast =0$ and thus \eqref{3} holds.

Applying \cref{lem3.1}(i) to $F$ at $(u,v)$ gives that
\begin{equation}\label{3.13}
	\widehat D^*F(u,v)(\mathbb{Y}^*)\subseteq \widehat{\mathbf{N}}(A_1,u-v(1)) +\widehat{\mathbf{N}}(A_2, u-v(2)).
\end{equation}
This together with \eqref{2} and \eqref{3} implies that
$$
\widehat D^*F(u,v)(\mathbb{Y}^*)\subseteq \mathbb{Z}^*\times [0,+\infty) +\{(0,0)\}.
$$
Noting that $\widehat{\mathbf{N}}(F^{-1}(\bar y), \bar x)=\mathbb{Z}^\ast \times \mathbb{R}$ and $\varepsilon\mathbf{B}_{\mathbb{X}^\ast }\subseteq \mathbb{Z}^\ast \times [-\epsilon, \epsilon]$, it follows that
$$
\mathbf{N}(F^{-1}(\bar y),\bar x)\not\subseteq \mathbb{Z}^*\times [0,+\infty) +\{(0,0)\}+\varepsilon\mathbf{B}_{\mathbb{X}^\ast },
$$
and thus \eqref{1} holds. The proof is complete.\end{proof}

\medskip

 {Note that {\it subtransversality} of finitely many closed sets is a well-known and important concept in mathematical programming and approximation theory. Given a collection of finitely many closed sets $\{A_1,\cdots, A_m\}$ in $\mathbb{X}$, recall that $\{A_1,\cdots, A_m\}$ is said to be {\it subtransversal} at $\bar x\in\bigcap_{i=1}^mA_i$, if there exist $\tau,\delta>0$ such that
\begin{equation}\label{3.27-250831}
{\bf d}\Big(x, \bigcap_{i=1}^mA_i\Big)\leq \tau\sum_{i=1}^{m}{\bf d}(x, A_i)\ \ \forall x\in {\bf B}(\bar x,\delta).
\end{equation}
 It is known from \cite[7.1.3 Comments]{Ioffe2017} that inequality \eqref{3.27-250831} was in reality introduced by Dolecki (see\cite{Dolecki1982}) in a very different context,
and consequently this property has been extensively studied by many authors under various different names; e.g., ``metric qualification condition" in \cite{Ioffe1989,IP1996}, ``linear coherence" in \cite[Theorem 4.7.5]{Pen13},  ``linear regularity" in \cite{BB1993,BB1996}, ``metric inequality" in \cite{NT2001} and so on.}

 {It is noted that subtransversality is closely related with metric subregularity. Let $\mathbb{X}^m$ be equipped with $\ell_1$ norm. We consider the multifunction $F:\mathbb{X}\rightrightarrows \mathbb{X}^m$ defined by
$$
F(x):=(A_1-x)\times\cdots\times (A_m-x)\ \ \forall x\in \mathbb{X}.
$$
Then it is easy to verify that the subtransversality property as sain in \eqref{3.27-250831} is equivalent to metric subregularity of $F$ at $(\bar x, 0_{\mathbb{X}^m})$. It is known from \cite{NT2001,WTY2024-SVVA} that subtransverality of finitely many closed sets could imply the strong limiting CHIP (expressed in terms of limiting normal cones) in the Asplund space and such implications, if to consider all collections of finitely many closed sets, are able to characterize Asplund spaces. Further, based on \cite[Theorem 3.8]{NT2001} and Theorem 3.1, the following theorem is to prove an equivalence result between subtransverality of finitely many closed sets and metric subregularity of multifunctions in the sense of deriving necessary dual conditions in terms of limiting normal cones; that is,}

 {\begin{theorem}
	Let $\mathbb{X}$ be a Banach space. The following statements are equivalent:
	\begin{itemize}
				\item[\rm (ii)] For every nonempty closed sets $A_1,\cdots,A_m$ in $\mathbb{X}$ being subtransversal at $\bar x\in \bigcap_{i=1}^mA_i$, $\{A_1,\cdots,A_m\}$ has the strong limiting CHIP at $\bar x$ (i.e., $N(\bigcap_{i=1}^mA_i,\bar x)\subseteq \sum_{i=1}^mN(A_i, \bar x)$); 
		\item[\rm (ii)] For every Asplund space $\mathbb{Y}$ and every $F \in \Gamma(\mathbb{X}, \mathbb{Y})$ being metrically subregular at $(\bar x, \bar y) \in {\rm gph}(F)$, $F$ satisfies the limiting $\mathbf{BCQ}$ at $(\bar x, \bar y)$.
	\end{itemize}
\end{theorem}}

 {Similarly, by virtue of \cite[Theorem 3.2]{WTY2024-SVVA} and \cref{theorem3.2},  the next thereom also provides an equivalence result on subtransversality and metric subregularity when dealing with necessary dual conditions for these properties via Fr\'echet normal cones.}

 {	\begin{theorem}
	Let $\mathbb{X}$ be a Banach space. Then the following statements are equivalent:
	\begin{itemize}
			\item [\rm(i)] For every nonempty closed sets $A_1,\cdots,A_m$ in $\mathbb{X}$ being subtransversal at $\bar x\in \bigcap_{i=1}^mA_i$, then for all $\varepsilon>0$, one has
			$$
			\widehat{N}\left(\bigcap_{i=1}^mA_i, \bar x\right)\subseteq \bigcup\left\{\sum_{i=1}^m \widehat{N}(A_i,  x_i): x_i\in A_i\cap{\bf B}(\bar x,\varepsilon),i=1,\cdots,m\right\} +\varepsilon {\bf B}_{\mathbb{X}^*}.
			$$
		\item [\rm(ii)] For every Asplund space $\mathbb{Y}$ and every $F\in\Gamma(\mathbb{X}, \mathbb{Y})$ being metrically subregular at $(\bar x,\bar y)\in{\rm gph}(F)$, then 
		for all $\varepsilon>0$, one has \eqref{3-16a} holds.
	\end{itemize}
\end{theorem}}

\medskip

\textit{In the remainder of this section, unless stated otherwise, we always assume that $\mathbb{X}$ and $\mathbb{Y}$ are Asplund spaces.}

\medskip

We note that \cref{theorem3.2} provides necessary conditions for the metric subregularity of closed multifunctions in terms of Fr\'echet normal cones and coderivatives in Asplund spaces. In fact, a sharper necessary condition for metric subregularity can be stated as follows.

\medskip

\begin{theorem}\label{th3.2}
Let $F\in\Gamma(\mathbb{X},\mathbb{Y})$ be such that $F$ is metrically subregular at $(\bar x,\bar y)\in{\rm gph}(F)$. Then there exist constants $\tau,\delta\in(0,+\infty)$ such that for any $\epsilon>0$,  {one has}
\begin{equation}\label{3.14}
\widehat{\mathbf{N}}(F^{-1}(\bar y), x)\cap \mathbf{B}_{\mathbb{X}^*} \subseteq 
\left\{ \tau \widehat D^*F(u,v)\big((1+\epsilon)\mathbf{B}_{\mathbb{Y}^*}\big) : (u,v)\in \mathbf{B}((x,\bar y),\epsilon)\cap {\rm gph}(F) \right\} + \epsilon \mathbf{B}_{\mathbb{X}^*}
\end{equation}
 {holds} for all $x\in \mathbf{B}(\bar x,\delta)\cap F^{-1}(\bar y)$.
\end{theorem}

\medskip

\noindent \textbf{Proof.} 
By metric subregularity, there exist $\tau, r \in (0,+\infty)$ such that \eqref{3.1} holds. Define a norm on $\mathbb{X}\times \mathbb{Y}$ by
\[
\|(x, y)\|_{\tau} := \frac{\tau+1}{\tau}\|x\| + \|y\|, \quad (x, y)\in \mathbb{X}\times \mathbb{Y}.
\] 
Then \eqref{3.6} holds with $\delta := r/2$.  

Let $\epsilon > 0$ be given. Take any $x\in \mathbf{B}(\bar x,\delta)\cap F^{-1}(\bar y)$ and $x^*\in \widehat{\mathbf{N}}(F^{-1}(\bar y), x)\cap \mathbf{B}_{\mathbb{X}^*}$. Using the argument in the proof of \eqref{3.8-250618}, we have
\[
\Big(\frac{x^*}{\tau}, 0\Big) \in \widehat\partial \big(\mathbf{d}_{\|\cdot\|_{\tau}}(\cdot, {\rm gph}(F)) + \phi \big)(x,\bar y),
\]
where $\phi(u,v) := \|v - \bar y\|$ for all $(u,v)\in \mathbb{X}\times \mathbb{Y}$.  

Choose $\epsilon_1\in (0,\epsilon)$ such that $2(\tau+1)\epsilon_1 < \epsilon$. By \cref{lem2.1}, there exist $(x_1,y_1), (x_2,y_2) \in \mathbf{B}(x,\epsilon_1)\times \mathbf{B}(\bar y,\epsilon_1)$ such that
\begin{align*}
\Big(\frac{x^*}{\tau},0\Big) &\in \widehat\partial \mathbf{d}_{\|\cdot\|_{\tau}}(\cdot, {\rm gph}(F))(x_1,y_1) + \widehat\partial \phi(x_2,y_2) + \epsilon_1(\mathbf{B}_{\mathbb{X}^*} \times \mathbf{B}_{\mathbb{Y}^*}) \\
&\subseteq \widehat\partial \mathbf{d}_{\|\cdot\|_{\tau}}(\cdot, {\rm gph}(F))(x_1,y_1) + \{0\}\times \mathbf{B}_{\mathbb{Y}^*} + \epsilon_1(\mathbf{B}_{\mathbb{X}^*} \times \mathbf{B}_{\mathbb{Y}^*}).
\end{align*}
Hence, there exist $(x_1^*,y_1^*) \in \widehat\partial \mathbf{d}_{\|\cdot\|_{\tau}}(\cdot, {\rm gph}(F))(x_1,y_1)$ and $b_1^* \in \mathbf{B}_{\mathbb{Y}^*}$ such that
\begin{equation}\label{3.16}
\Big(\frac{x^*}{\tau},0\Big) \in (x_1^*,y_1^*) + (0,b_1^*) + \epsilon_1(\mathbf{B}_{\mathbb{X}^*} \times \mathbf{B}_{\mathbb{Y}^*}).
\end{equation}
By \cref{lem2.2}, there exist $(u_1,v_1) \in {\rm gph}(F)$ and $(u_1^*,v_1^*) \in \widehat{\mathbf{N}}({\rm gph}(F),(u_1,v_1))$ such that
\begin{equation}\label{3.17}
\|(u_1,v_1)-(x_1,y_1)\| < \mathbf{d}_{\|\cdot\|_{\tau}}((u_1,v_1),{\rm gph}(F)) + \epsilon_1, 
\quad
\|(u_1^*,v_1^*)-(x_1^*,y_1^*)\| < \epsilon_1.
\end{equation}
Combining \eqref{3.16} and \eqref{3.17}, we obtain
\begin{equation}\label{3.18}
\Big(\frac{x^*}{\tau},0\Big) \in (u_1^*,v_1^*) + (0,b_1^*) + 2\epsilon_1 (\mathbf{B}_{\mathbb{X}^*} \times \mathbf{B}_{\mathbb{Y}^*}).
\end{equation}
Note that
\[
(x_1^*,y_1^*) \in \widehat\partial \mathbf{d}_{\|\cdot\|_{\tau}}(\cdot, {\rm gph}(F))(x_1,y_1) \subseteq \frac{\tau+1}{\tau}\mathbf{B}_{\mathbb{X}^*} \times \mathbf{B}_{\mathbb{Y}^*},
\]
and it follows from \eqref{3.17} that
\[
(u_1,v_1) \in \mathbf{B}((x,\bar y),\epsilon), \quad 
u_1^* \in \widehat D^*F(u_1,v_1)(-v_1^*), \quad \|v_1^*\|\le 1+2\epsilon_1.
\]
Finally, from \eqref{3.16} and \eqref{3.18}, we get
\[
x^* \in \tau \widehat D^*F(u_1,v_1)\big((1+2\epsilon_1)\mathbf{B}_{\mathbb{Y}^*}\big) + 2\tau \epsilon_1 \mathbf{B}_{\mathbb{X}^*} 
\subseteq \tau \widehat D^*F(u_1,v_1)\big((1+\epsilon)\mathbf{B}_{\mathbb{Y}^*}\big) + \epsilon \mathbf{B}_{\mathbb{X}^*}.
\]
Hence, \eqref{3.14} holds. \hfill $\Box$

\medskip

It is well-known that main calculus results expressed via Fr\'echet constructions hold only in a \emph{fuzzy} form; thus, $\epsilon$ in \eqref{3.14} cannot generally be taken as $0$. However, in finite-dimensional spaces, one can pass to the limit as $\epsilon \downarrow 0$ to obtain necessary conditions for metric subregularity in terms of \emph{limiting} normal cones and coderivatives, as shown in the following theorem.

\begin{theorem}\label{th3.3}
Let $\mathbb{X}$ be finite-dimensional and $F\in\Gamma(\mathbb{X},\mathbb{Y})$ be metrically subregular at $(\bar x,\bar y)\in{\rm gph}(F)$. Then there exist constants $\tau, \delta \in (0,+\infty)$ such that
\begin{equation}\label{3.18a}
\mathbf{N}(F^{-1}(\bar y), x)\cap \mathbf{B}_{\mathbb{X}^*} \subseteq \tau D^*F(x,\bar y)(\mathbf{B}_{\mathbb{Y}^*})
\end{equation}
 {holds} for all $x \in \mathbf{B}(\bar x,\delta)\cap F^{-1}(\bar y)$.
\end{theorem}

{\bf Proof.} By the metric subregularity of $F$, there exist $\tau,r\in(0,+\infty)$ such that \eqref{3.1} holds. Following the proof of \cref{th3.1}, define 
\[
\|(x, y)\|_{\tau} := \frac{\tau+1}{\tau}\|x\| + \|y\|, \quad \forall (x, y) \in \mathbb{X}\times \mathbb{Y}.
\] 
Then \eqref{3.6} holds with $\delta := \frac{r}{2}$.

Let $x \in \mathbf{B}(\bar x,\delta) \cap F^{-1}(\bar y)$ and $x^* \in {\mathbf{N}}(F^{-1}(\bar y), x) \cap \mathbf{B}_{\mathbb{X}^*}$. By the definition of the limiting normal cone, there exist sequences $x_n \xrightarrow{F^{-1}(\bar y)} x$ and $x_n^* \stackrel{\|\cdot\|}{\longrightarrow} x^*$ with $x_n^* \in \widehat{\mathbf{N}}(F^{-1}(\bar y), x_n)$ for all $n$. Without loss of generality, assume $x_n^* \neq 0$ and set $\widetilde{x_n^*} := \frac{x_n^*}{\|x_n^*\|}$. Then 
\[
\widetilde{x_n^*} \in \widehat{\mathbf{N}}(F^{-1}(\bar y), x_n) \cap \mathbf{B}_{\mathbb{X}^*} = \widehat{\partial}\mathbf{d}(\cdot, F^{-1}(\bar y))(x_n).
\]

Using the proof of \eqref{3.8-250618} and applying \cref{lem2.1}, for each $n$ there exist 
\[
(u_n,v_n) \in \mathbf{B}\Big(x_n, \frac{1}{n}\Big) \times \mathbf{B}\Big(\bar y, \frac{1}{n}\Big), \quad (u_n^*,v_n^*) \in \widehat\partial \mathbf{d}_{\|\cdot\|_{\tau}}(\cdot, \mathrm{gph}(F))(u_n,v_n), \quad b_n^* \in \mathbf{B}_{\mathbb{Y}^*}
\] 
such that
\begin{equation}\label{3.19-polished}
\Big(\frac{\widetilde{x_n^*}}{\tau},0\Big) \in (u_n^*,v_n^*) + (0,b_n^*) + \frac{1}{n} (\mathbf{B}_{\mathbb{X}^*} \times \mathbf{B}_{\mathbb{Y}^*}).
\end{equation}

By \cref{lem2.2}, for each $(u_n^*,v_n^*) \in \widehat\partial \mathbf{d}_{\|\cdot\|_{\tau}}(\cdot, \mathrm{gph}(F))(u_n,v_n)$, there exist 
\[
(\widetilde{u_n}, \widetilde{v_n}) \in \mathrm{gph}(F), \quad (\widetilde{u_n^*}, \widetilde{v_n^*}) \in \widehat{\mathbf{N}}(\mathrm{gph}(F), (\widetilde{u_n}, \widetilde{v_n}))
\] 
such that
\begin{equation}\label{3.20-polished}
\|(\widetilde{u_n}, \widetilde{v_n}) - (u_n,v_n)\| < \mathbf{d}_{\|\cdot\|_{\tau}}((u_n,v_n), \mathrm{gph}(F)) + \frac{1}{n}, \quad 
\|(\widetilde{u_n^*}, \widetilde{v_n^*}) - (u_n^*,v_n^*)\| < \frac{1}{n}.
\end{equation}
Combining \eqref{3.19-polished} and \eqref{3.20-polished} gives
\begin{equation}\label{3.21-polished}
\Big(\frac{\widetilde{x_n^*}}{\tau},0\Big) \in (\widetilde{u_n^*}, \widetilde{v_n^*}) + (0,b_n^*) + \frac{2}{n} (\mathbf{B}_{\mathbb{X}^*} \times \mathbf{B}_{\mathbb{Y}^*}).
\end{equation}

Since $\mathbb{X}$ is finite-dimensional and $\mathbb{Y}$ is Asplund, and
\[
(u_n^*,v_n^*) \in \widehat\partial \mathbf{d}_{\|\cdot\|_{\tau}}(\cdot, \mathrm{gph}(F))(u_n,v_n) \subseteq \frac{\tau+1}{\tau} \mathbf{B}_{\mathbb{X}^*} \times \mathbf{B}_{\mathbb{Y}^*},
\]
we can assume (passing to a subsequence if necessary) that
\[
u_n^* \stackrel{\|\cdot\|}{\longrightarrow} u^*, \quad v_n^* \stackrel{w^*}{\longrightarrow} v^*, \quad b_n^* \stackrel{w^*}{\longrightarrow} b^* \in \mathbf{B}_{\mathbb{Y}^*}.
\] 
Then \eqref{3.20-polished} implies
\[
(\widetilde{u_n}, \widetilde{v_n})\xrightarrow {\mathrm{gph}(F)} (x,\bar y), \quad \widetilde{u_n^*} \stackrel{\|\cdot\|}{\longrightarrow} u^*, \quad \widetilde{v_n^*} \stackrel{w^*}{\longrightarrow} v^*,
\]
so that $(u^*,v^*) \in \mathbf{N}(\mathrm{gph}(F),(x,\bar y))$. Taking the limit in \eqref{3.21-polished} yields
\[
\Big(\frac{x^*}{\tau \|x^*\|}, 0\Big) = (u^*, v^*) + (0,b^*),
\]
and consequently
\[
x^* \in \tau D^*F(x,\bar y)(\|x^*\| b^*) \subseteq \tau D^*F(x,\bar y)(\mathbf{B}_{\mathbb{Y}^*}),
\]
since $\|x^*\| \leq 1$. Hence, \eqref{3.18a} holds. \hfill $\Box$
\medskip

\begin{remark}
\cref{th3.3} provides a necessary condition for the metric subregularity of $F$ at $(\bar x,\bar y)$, as stated in \eqref{3.18a}. However, \eqref{3.18a} may not, in general, be sufficient for the metric subregularity of $F$ at $(\bar x,\bar y)$, even in finite-dimensional spaces. The following example illustrates this.

Let $\mathbb{X}=\mathbb{Y}=\mathbb{R}$ and define the function $f:\mathbb{X}\to\mathbb{R}$ by  
\begin{equation*}
	f(x):=
	\begin{cases}
		x, & x\leq 0,\\
		x^2, & x> 0.
	\end{cases}
\end{equation*}
Define $F:\mathbb{X}\rightrightarrows\mathbb{Y}$ by $F(x) = [f(x), +\infty)$ for any $x\in\mathbb{X}$, and take $\bar x=\bar y=0$. One can verify that  
\[
F^{-1}(\bar y) = (-\infty,0], \quad D^*F(\bar x,\bar y)(1) \supseteq \{0,1\},
\]
so \eqref{3.18a} holds for $x=\bar x$ with $\tau=1$.

On the other hand, if $x_k := 1/k$ for $k\in\mathbb{N}$, then  
\[
\frac{\mathbf{d}(x_k, F^{-1}(\bar y))}{\mathbf{d}(0, F(x_k))} = k \ \longrightarrow\ +\infty \quad \text{as} \quad k\to\infty.
\]
This shows that $F$ is not metrically subregular at $(\bar x,\bar y)$.
\end{remark}

\medskip

We now study metric subregularity for composite-convex multifunctions in the setting of Asplund spaces. From \cref{th3.1,th3.2,th3.3}, we know that necessary conditions for the metric subregularity of closed multifunctions can be expressed in terms of Fr\'echet and limiting normal cones, as well as coderivatives. Ideally, such conditions would also be sufficient for metric subregularity. In fact, this ideal case occurs for composite-convex multifunctions in Asplund spaces.

\medskip

The next theorem provides dual characterizations of metric subregularity for composite-convex multifunctions in Asplund spaces, expressed in terms of Fr\'echet and limiting normal cones, together with coderivatives.

\medskip

\begin{theorem}\label{th3.4}
Let $\mathbb{E}$ be an Asplund space,  $g:\mathbb{X}\to\mathbb{E}$ be a continuously differentiable mapping, and let $G\in\Gamma(\mathbb{E},\mathbb{Y})$ be such that ${\rm gph}(G)$ is convex. Set $F := G\circ g$ and let $\bar x\in\mathbb{X}$ be such that $\nabla g(\bar x)$ is surjective, and $\bar y\in F(\bar x)$. Then the following statements are equivalent:
\begin{itemize}
    \item[\rm(i)] $F$ is metrically subregular at $(\bar x,\bar y)$.
    \item[\rm(ii)] There exist $\tau,\delta>0$ such that
    \begin{equation}\label{3.23}
        \widehat{\mathbf{N}}(F^{-1}(\bar y), x) \cap \mathbf{B}_{\mathbb{X}^*} \subseteq \tau\, \widehat D^*F(x,\bar y)(\mathbf{B}_{\mathbb{Y}^*})
    \end{equation}
    holds for all $x \in \mathbf{B}(\bar x,\delta) \cap F^{-1}(\bar y)$.
    \item[\rm(iii)] There exist $\tau,\delta>0$ such that
    \begin{equation}\label{3.24}
        \mathbf{N}(F^{-1}(\bar y), x) \cap \mathbf{B}_{\mathbb{X}^*} \subseteq \tau\, D^*F(x,\bar y)(\mathbf{B}_{\mathbb{Y}^*})
    \end{equation}
    holds for all $x \in \mathbf{B}(\bar x,\delta) \cap F^{-1}(\bar y)$.
\end{itemize}
\end{theorem}
\noindent\textbf{Proof.} 
Define $\Psi(x,y) := (g(x),y)$ for all $(x,y) \in \mathbb{X} \times \mathbb{Y}$. Then  
\[
\nabla \Psi(\bar{x}, \bar{y}) = \big(\nabla g(\bar{x}), \mathbf{I}_{\mathbb{Y}^*}\big),
\]
where $\mathbf{I}_{\mathbb{Y}^*}$ is the identity operator on $\mathbb{Y}^*$.  
It is straightforward to verify that $\nabla \Psi(\bar{x}, \bar{y})$ is surjective.  

Note that $\mathrm{gph}(F) = \Psi^{-1}(\mathrm{gph}(G))$. By virtue of \cref{lem3.2}, there exist $\ell, L, r \in (0, +\infty)$ such that
\begin{equation}\label{3.25}
\begin{aligned}
\widehat{\mathbf{N}}\big(\mathrm{gph}(F),(x,y)\big) 
\cap \ell\big(\mathbf{B}_{\mathbb{X}^*} \times \mathbf{B}_{\mathbb{Y}^*}\big) 
&\subseteq \nabla \Psi(x,y)^* \big( \mathbf{N}(\mathrm{gph}(G), \Psi(x,y)) \big) 
\cap \big( \mathbf{B}_{\mathbb{E}^*} \times \mathbf{B}_{\mathbb{Y}^*} \big) \\
&\subseteq \widehat{\mathbf{N}}\big(\mathrm{gph}(F),(x,y)\big) 
\cap L\big(\mathbf{B}_{\mathbb{X}^*} \times \mathbf{B}_{\mathbb{Y}^*}\big)
\end{aligned}
\end{equation}
for all $(x,y) \in \big(\mathbf{B}(\bar{x},r) \times \mathbf{B}(\bar{y},r)\big) \cap \mathrm{gph}(F)$.  

Moreover, there exist $\ell_1, L_1 \in (0, +\infty)$ such that
\begin{equation}\label{3.25a}
\widehat{\mathbf{N}}\big(F^{-1}(\bar{y}),u\big) 
\cap \ell_1 \mathbf{B}_{\mathbb{X}^*} 
\subseteq \nabla g(u)^* \big( \mathbf{N}(G^{-1}(\bar{y}), g(u)) \big) 
\cap \mathbf{B}_{\mathbb{E}^*}
\subseteq \widehat{\mathbf{N}}\big(F^{-1}(\bar{y}),u\big) 
\cap L_1 \mathbf{B}_{\mathbb{X}^*}
\end{equation}
for all $u \in \mathbf{B}(\bar{x},r) \cap F^{-1}(\bar{y})$ (shrinking $r>0$ if necessary).

%
\medskip	
We first show the equivalence of (i) and (ii). 

\medskip 	

\underline{(i) $\Rightarrow$ (ii)}:  
Suppose that $F$ is metrically subregular at $(\bar{x}, \bar{y})$.  
Then, by \cref{th3.2}, there exist $\tau, \delta > 0$ such that \eqref{3.14} holds for all  
$
x \in \mathbf{B}(\bar{x}, \delta) \cap F^{-1}(\bar{y}).$  
Fix $x \in \mathbf{B}(\bar{x}, \delta) \cap F^{-1}(\bar{y})$ and take  
$
x^* \in \widehat{\mathbf{N}}(F^{-1}(\bar{y}), x) \cap \mathbf{B}_{\mathbb{X}^*}. $
By \eqref{3.14}, for each $n \in \mathbb{N}$, there exist  
\[
(x_n, y_n) \xrightarrow{{\rm gph}(F)}(x, \bar{y}), 
\quad (x_n^*, y_n^*) \in \mathbb{X}^* \times \Big(1+\tfrac{1}{n}\Big)\mathbf{B}_{\mathbb{Y}^*}, 
\quad b_n^* \in \mathbf{B}_{\mathbb{Y}^*},
\]  
such that $(x_n^*, y_n^*) \in \widehat{\mathbf{N}}({\rm gph}(F), (x_n, y_n))$ and  
\begin{equation}\label{3.26}
    x^* = \tau x_n^* + \tfrac{1}{n} b_n^*.
\end{equation}
Note that $\mathrm{gph}(F) = \Psi^{-1}(\mathrm{gph}(G))$.  
It then follows from \cite[Corollary~1.15]{Mordukhovich} that  
\[
 \widehat{\mathbf{N}}(\mathrm{gph}(F), (x_n, y_n)) 
= \nabla \Psi(x_n, y_n)^* \big( \mathbf{N}(\mathrm{gph}(G), (g(x_n), y_n)) \big).
\]  
Thus, there exists $(e_n^*, w_n^*) \in \mathbf{N}(\mathrm{gph}(G), (g(x_n), y_n))$ such that  
\begin{equation}\label{3.27}
    (x_n^*, y_n^*) 
    = \nabla \Psi(x_n, y_n)^*(e_n^*, w_n^*) 
    = \big( \nabla g(x_n)^*(e_n^*), \, w_n^* \big).
\end{equation}

From \eqref{3.26}, there exists $M>0$ such that $\|(x_n^*, y_n^*)\| \le M$ for all $n$.  
Then, by \eqref{3.25}, there exists  
\[
(\tilde{e}_n^*,\tilde{w}_n^*) \in \mathbf{N}\big(\mathrm{gph}(G),(g(x_n),y_n)\big) 
\cap \big(\mathbf{B}_{\mathbb{E}^*} \times \mathbf{B}_{\mathbb{Y}^*}\big)
\]
such that
\[
\frac{\ell}{M}(x_n^*, y_n^*) 
= \nabla\Psi(x_n,y_n)^*(\tilde{e}_n^*,\tilde{w}_n^*) 
= \big(\nabla g(x_n)^*(\tilde{e}_n^*), \tilde{w}_n^*\big).
\]
Combining this with \eqref{3.27} yields
\[
y_n^* = w_n^* = \frac{M}{\ell} \tilde{w}_n^*, 
\quad 
x_n^* = \nabla g(x_n)^*(e_n^*) 
= \nabla g(x_n)^*\!\left(\frac{M}{\ell} \tilde{e}_n^*\right),
\]
and hence
\[
\|e_n^*\| 
= \left\| \frac{M}{\ell} \tilde{e}_n^* \right\| 
\le \frac{M}{\ell},
\]
since $\nabla g(x_n)^*$ is one-to-one.

Because $\mathbb{X}$ and $\mathbb{Y}$ are Asplund spaces, without loss of generality we may assume that  
\begin{equation}\label{3.28}
(e_n^*,w_n^*) \xrightarrow{w^*} (e^*,w^*) \in \mathbb{X}^* \times \mathbb{Y}^*
\end{equation}
(possibly after passing to a subsequence).  
Note that $\|w_n^*\| = \|y_n^*\| \le 1+\frac{1}{n}$, hence
\[
\|w^*\| \le \limsup_{n \to \infty} \|w_n^*\| \le 1,
\]
which means $w^* \in \mathbf{B}_{\mathbb{Y}^*}$.  

Since $(e_n^*,w_n^*) \in \mathbf{N}\big(\mathrm{gph}(G),(g(x_n),y_n)\big)$ and $\mathrm{gph}(G)$ is convex, we have
\[
(e^*,w^*) \in \mathbf{N}\big(\mathrm{gph}(G),(g(x),\bar{y})\big),
\]
because $ {(x_n,y_n)\xrightarrow{\mathrm{gph}(F)} (x,\bar{y})}$.  
Together with \eqref{3.27} and \eqref{3.28}, this implies
\[
(x_n^*, y_n^*) \xrightarrow{w^*} \big(\nabla g(x)^*(e^*), w^*\big) 
= \nabla\Psi(x,y)^*(e^*,w^*) 
\in \nabla\Psi(x,y)^*\big( \mathbf{N}(\mathrm{gph}(G), \Psi(x,\bar{y})) \big).
\]
Using \cite[Corollary~1.15]{Mordukhovich} again, we conclude that
\[
\big(\nabla g(x)^*(e^*), w^*\big) \in \widehat{\mathbf{N}}\big(\mathrm{gph}(F), (x,\bar{y})\big).
\]
Finally, taking weak$^*$ limits in \eqref{3.26} as $n \to \infty$ yields
\[
x^* = \tau \nabla g(x)^*(e^*) 
\in \tau\, \widehat{D}^*F(x,\bar{y})\big( \mathbf{B}_{\mathbb{Y}^*} \big),
\]
since $\|w^*\| \le 1$.  
Hence, \eqref{3.23} holds.

\medskip

 \underline{(ii) $\,\Rightarrow\,$ (i)}:  
Suppose there exist $\tau,\delta \in (0,+\infty)$ such that \eqref{3.23} holds.  
Let $\epsilon > 0$ be such that $(\tau+1)\epsilon < \ell$.  
Since $g$ is continuously differentiable, there exists $\delta_1 > 0$ with  
$\delta_1 < \min\{\delta,r\}$ such that
\begin{equation}\label{3.29}
\|g(x) - g(u) - \nabla g(u)(x-u)\| < \epsilon \|x-u\|,  
\quad \forall x,u \in \mathbf{B}(\bar{x},\delta_1).
\end{equation}
Set $r_1 := \delta_1/2$ and take any  
$x \in \mathbf{B}(\bar{x},r_1) \setminus F^{-1}(\bar{y})$.  
Then
\[
\mathbf{d}\big(x, F^{-1}(\bar{y})\big) \le \|x-\bar{x}\| < r_1.
\]
Choose $\beta \in (0,1)$ such that
\begin{equation}\label{3.30}
\beta > \max\left\{
\frac{\mathbf{d}(x,F^{-1}(\bar{y}))}{r_1}, \;
\frac{(\tau+1)\epsilon}{\ell}
\right\}.
\end{equation}
By \cref{lem2.3}, there exist  
$u \in F^{-1}(\bar{y})$ and $u^* \in \widehat{\mathbf{N}}\big(F^{-1}(\bar{y}),u\big)$  
with $\|u^*\| = 1$ such that
\begin{equation}\label{3.31}
\beta\|x-u\| \le 
\min\left\{ \mathbf{d}\big(x, F^{-1}(\bar{y})\big), \;
\langle u^*, x-u\rangle \right\}.
\end{equation}
It follows that
\[
\|u - \bar{x}\| \le \|u-x\| + \|x-\bar{x}\| 
< \frac{\mathbf{d}(x,F^{-1}(\bar{y}))}{\beta} + r_1
< 2r_1 = \delta_1.
\]
Hence, by \eqref{3.23}, there exists $b^* \in \mathbf{B}_{\mathbb{Y}^*}$ such that
\[
\left( \frac{u^*}{\tau}, b^* \right) 
\in \widehat{\mathbf{N}}\big(\mathrm{gph}(F),(u,\bar{y})\big).
\]
From \eqref{3.25}, there exists  
$(e^*,y^*) \in \mathbf{N}\big(\mathrm{gph}(G),(g(u),\bar{y})\big)  
\cap \big( \mathbf{B}_{\mathbb{E}^*} \times \mathbf{B}_{\mathbb{Y}^*} \big)$  
such that
\begin{equation}\label{3.32}
\frac{\tau\ell}{\tau+1} \left( \frac{u^*}{\tau}, b^* \right)  
= \nabla\Psi(u,\bar{y})^*(e^*,y^*)  
= \big( \nabla g(u)^*(e^*), y^* \big).
\end{equation}

For any $y \in F(x)$, we have $(g(x), y) \in \mathrm{gph}(G)$, and thus
\[
\langle (e^*,y^*), \, (g(x)-g(u),\, y-\bar{y}) \rangle \le 0.
\]
Combining this with \eqref{3.29} and \eqref{3.32}, we obtain
\begin{align*}
\left\langle \frac{\ell u^*}{\tau+1}, x-u \right\rangle
&= \langle \nabla g(u)^*(e^*), x-u \rangle \\
&= \langle e^*, \nabla g(u)(x-u) \rangle \\
&= \langle e^*, g(u) - g(x) + \nabla g(u)(x-u) \rangle 
  + \langle e^*, g(x) - g(u) \rangle \\
&\le \|g(u) - g(x) + \nabla g(u)(x-u)\| 
  + \langle -y^*, y - \bar{y} \rangle \\
&< \epsilon \|x-u\| + \|y - \bar{y}\|.
\end{align*}
Therefore,
\[
\left\langle \frac{\ell u^*}{\tau+1}, x-u \right\rangle
< \epsilon \|x-u\| + \mathbf{d}(\bar{y}, F(x)).
\]
By \eqref{3.31}, this implies
\[
\beta\|x-u\| 
\le \langle u^*, x-u \rangle 
< \frac{(\tau+1)\epsilon}{\ell} \|x-u\|  
+ \frac{\tau+1}{\ell} \mathbf{d}(\bar{y},F(x)).
\]
Hence,
\[
\left( \beta - \frac{(\tau+1)\epsilon}{\ell} \right) 
\mathbf{d}\big(x, F^{-1}(\bar{y})\big) 
< \frac{\tau+1}{\ell} \mathbf{d}(\bar{y},F(x)).
\]
Taking the limit as $\beta \uparrow 1$, we obtain
\[
\mathbf{d}\big(x, F^{-1}(\bar{y})\big) 
\le \frac{\tau+1}{\ell - (\tau+1)\epsilon} \, \mathbf{d}(\bar{y},F(x)).
\]
Thus $F$ is metrically subregular at $(\bar{x},\bar{y})$ with constant  
$\frac{\tau+1}{\ell - (\tau+1)\epsilon} > 0$.

For the equivalence of (ii) and (iii), we claim that
\begin{equation}\label{3.34}
\widehat{\mathbf{N}}\big(\mathrm{gph}(F),(x,y)\big) 
= \mathbf{N}\big(\mathrm{gph}(F),(x,y)\big),  
\quad \forall (x,y) \in 
\big( \mathbf{B}(\bar{x},r) \times \mathbf{B}(\bar{y},r) \big)  
\cap \mathrm{gph}(F),
\end{equation}
and
\begin{equation}\label{3.35}
\widehat{\mathbf{N}}\big(F^{-1}(\bar{y}),u\big) 
= \mathbf{N}\big(F^{-1}(\bar{y}),u\big),  
\quad \forall u \in \mathbf{B}(\bar{x},r) \cap F^{-1}(\bar{y}).
\end{equation}

Let $(x,y) \in \left( \mathbf{B}(\bar x,r) \times \mathbf{B}(\bar y,r) \right) \cap {\rm gph}(F)$ 
and $(x^*,y^*) \in \mathbf{N}({\rm gph}(F),(x,y))$.  
Then there exist sequences $  {(x_k,y_k)\xrightarrow{{\rm gph}(F)} (x,y)}$ 
and $(x_k^*,y_k^*) \stackrel{w^*}{\longrightarrow} (x^*,y^*)$ such that
\[
(x_k^*,y_k^*) \in \widehat{\mathbf{N}}({\rm gph}(F),(x_k,y_k)), 
\quad \forall k.
\]
By the Banach--Steinhaus theorem, we may assume that 
$\|(x_k^*,y_k^*)\| \leq M$ for some $M>0$ and all $k$.  
Hence
\[
\frac{\ell}{M}(x_k^*,y_k^*) 
   \in \widehat{\mathbf{N}}({\rm gph}(F),(x_k,y_k))
   \cap \ell\big( \mathbf{B}_{\mathbb{X}^*} \times \mathbf{B}_{\mathbb{Y}^*} \big).
\]
By virtue of \eqref{3.25}, for all $k$ sufficiently large, there exists 
\[
(e_k^*,v_k^*) \in 
\mathbf{N}({\rm gph}(G),(g(x_k),y_k)) 
\cap (\mathbf{B}_{\mathbb{E}^*} \times \mathbf{B}_{\mathbb{Y}^*})
\]
such that
\begin{equation}\label{3.36}
  \frac{\ell}{M}(x_k^*,y_k^*) 
    = \nabla\Psi(x_k,y_k)^*(e_k^*,v_k^*).
\end{equation}
Since $\mathbb{E}$ and $\mathbb{Y}$ are Asplund spaces, we may assume, 
without loss of generality, that 
\[
(e_k^*,v_k^*) \stackrel{w^*}{\longrightarrow} (e^*,v^*) 
\in \mathbf{B}_{\mathbb{E}^*} \times \mathbf{B}_{\mathbb{Y}^*}
\]
(possibly after passing to a subsequence).  
Because ${\rm gph}(G)$ is convex, one has 
\[
(e^*,v^*) \in \mathbf{N}({\rm gph}(G),(g(x),y)).
\]
Noting that $\nabla\Psi(x,y)^*$ is weak$^*$--weak$^*$ continuous 
and $\nabla\Psi$ is continuous, it follows that
\[
\nabla\Psi(x_k,y_k)^*(e_k^*,v_k^*) 
   \stackrel{w^*}{\longrightarrow} 
\nabla\Psi(x,y)^*(e^*,v^*).
\]
Taking the limit in \eqref{3.36} as $k \to \infty$ in the weak$^*$ topology yields
\[
\begin{aligned}
\frac{\ell}{M}(x^*,y^*) 
   &= \nabla\Psi(x,y)^*(e^*,v^*) \\
   &\in \nabla\Psi(x,y)^*\big( \mathbf{N}({\rm gph}(G),(g(x),y)) \big) \\
   &= \widehat{\mathbf{N}}({\rm gph}(F),(x,y)),
\end{aligned}
\]
where the last equality follows from \cite[Corollary~1.15]{Mordukhovich}.  
Thus $(x^*,y^*) \in \widehat{\mathbf{N}}({\rm gph}(F),(x,y))$, 
and hence \eqref{3.34} holds.

\medskip
For \eqref{3.35}, let $u \in \mathbf{B}(\bar x,r) \cap F^{-1}(\bar y)$ 
and $u^* \in \mathbf{N}(F^{-1}(\bar y),u)$.  
Then there exist $u_k \to u$ with $u_k \in F^{-1}(\bar y)$ 
and $u_k^* \stackrel{w^*}{\longrightarrow} u^*$ such that
\[
u_k^* \in \widehat{\mathbf{N}}(F^{-1}(\bar y),u_k), 
\quad \forall k.
\]
By the Banach--Steinhaus theorem, we may assume $\|u_k^*\| \leq M_1$ 
for some $M_1>0$ and all $k$.  
Hence
\[
\frac{\ell_1}{M_1}u_k^* 
   \in \widehat{\mathbf{N}}(F^{-1}(\bar y),u_k) 
   \cap \ell_1\mathbf{B}_{\mathbb{X}^*}.
\]
By \eqref{3.25a}, for $k$ sufficiently large, there exists 
\[
w_k^* \in \mathbf{N}(G^{-1}(\bar y), g(u_k)) 
        \cap \mathbf{B}_{\mathbb{E}^*}
\]
such that
\begin{equation}\label{3.37}
  \frac{\ell_1}{M_1}u_k^* 
     = \nabla g(u_k)^*(w_k^*).
\end{equation}
Since $\mathbb{E}$ is Asplund, we may assume (after possibly taking a subsequence) that
\[
w_k^* \stackrel{w^*}{\longrightarrow} w^* \in \mathbf{B}_{\mathbb{E}^*}.
\]
As $G^{-1}(\bar y)$ is convex, we have 
\[
w^* \in \mathbf{N}(G^{-1}(\bar y), g(u)).
\]
Since $\nabla g(u)^*$ is weak$^*$--weak$^*$ continuous 
and $\nabla g$ is continuous, it follows that
\[
\begin{aligned}
\frac{\ell_1}{M_1}u^* 
   &= \nabla g(u)^*(w^*) \\
   &\in \nabla g(u)^*\big( \mathbf{N}(G^{-1}(\bar y), g(u)) \big) \\
   &= \widehat{\mathbf{N}}(F^{-1}(\bar y),u),
\end{aligned}
\]
where the last equality follows from \cite[Corollary 1.15]{Mordukhovich}).  
Thus $u^* \in \widehat{\mathbf{N}}(F^{-1}(\bar y),u)$, and \eqref{3.35} holds.  
The proof is complete. \hfill$\blacksquare$ 

\medskip

\begin{theorem}\label{th3.5}
Let $\mathbb{E}$ be an Asplund space, 
$g:\mathbb{X} \to \mathbb{E}$ a continuously differentiable mapping, 
and $G \in \Gamma(\mathbb{E},\mathbb{Y})$ a convex multifunction.  
Let $F := G \circ g$ and $\bar x \in \mathbb{X}$ be such that 
$\nabla g(\bar x)$ is surjective and $\bar y \in F(\bar x)$.  
Then $F$ is metrically subregular at $(\bar x,\bar y)$ 
if and only if $G$ is metrically subregular at $(g(\bar x),\bar y)$.
\end{theorem}
\begin{proof}
Let $\Psi(x,y) := (g(x), y)$ for any $(x,y) \in \mathbb{X} \times \mathbb{Y}$. By \cref{lem3.2}, there exist constants $\ell, L, \ell_1, L_1, r \in (0, +\infty)$ such that
\eqref{3.25} holds for all $(x,y) \in (\mathbf{B}(\bar x,r) \times \mathbf{B}(\bar y,r)) \cap {\rm gph}(F)$ and \eqref{3.25a} holds for all $u \in \mathbf{B}(\bar x,r) \cap F^{-1}(\bar y)$ (possibly taking a smaller $r>0$ if necessary).

Based on \cref{th3.4}, it suffices to prove that there exist $\tau, \delta > 0$ such that \eqref{3.23} holds for all $x \in \mathbf{B}(\bar x,\delta) \cap F^{-1}(\bar y)$ if and only if there exist $\tau_1, \delta_1 > 0$ such that
\begin{equation}\label{3.38}
\mathbf{N}(G^{-1}(\bar y), e) \cap \mathbf{B}_{\mathbb{E}^*} \subseteq \tau_1 D^*G(e,\bar y)(\mathbf{B}_{\mathbb{Y}^*}), \quad \forall e \in \mathbf{B}(g(\bar x), \delta_1) \cap G^{-1}(\bar y).
\end{equation}

Suppose that there exist $\tau, \delta > 0$ such that \eqref{3.23} holds for all $x \in \mathbf{B}(\bar x, \delta) \cap F^{-1}(\bar y)$. Without loss of generality, assume $\delta < r$ (possibly taking a smaller $\delta$ if necessary). Since $\nabla g(\bar x)$ is surjective, it follows from \cite[Theorem 1.57]{Mordukhovich} that $g$ is metrically regular at $\bar x$, and thus there exist $\kappa, r_0 > 0$ such that
\begin{equation}\label{3.39}
\mathbf{d}(x, g^{-1}(e)) \leq \kappa \|e - g(x)\|, \quad \forall (x,e) \in \mathbf{B}(\bar x,r_0) \times \mathbf{B}(g(\bar x), r_0).
\end{equation}

Take $\delta_1 > 0$ such that $\kappa \delta_1 < \min\{\delta, r_0\}$. Let $e \in \mathbf{B}(g(\bar x), \delta_1) \cap G^{-1}(\bar y)$. Then
\[
\mathbf{d}(\bar x, g^{-1}(e)) \leq \kappa \|e - g(\bar x)\| < \kappa \delta_1 < \delta,
\]
so there exists $x \in g^{-1}(e)$ such that $x \in \mathbf{B}(\bar x, \delta) \cap F^{-1}(\bar y)$. By \eqref{3.25a}, we have
\[
\nabla g(x)^*(\mathbf{N}(G^{-1}(\bar y), e) \cap \mathbf{B}_{\mathbb{E}^*}) = \nabla g(x)^*(\mathbf{N}(G^{-1}(\bar y), g(x)) \cap \mathbf{B}_{\mathbb{E}^*}) \subseteq \widehat{\mathbf{N}}(F^{-1}(\bar y), x) \cap L_1 \mathbf{B}_{\mathbb{X}^*}.
\]

Let $e^* \in \mathbf{N}(G^{-1}(\bar y), e) \cap \mathbf{B}_{\mathbb{E}^*}$. Then \eqref{3.23} implies there exists $y^* \in \mathbf{B}_{\mathbb{Y}^*}$ such that 
\[
\Big(\frac{1}{L_1 \tau} \nabla g(x)^*(e^*), y^*\Big) \in \widehat{\mathbf{N}}({\rm gph}(F), (x, \bar y)).
\]
By \eqref{3.25}, there exists $(e_1^*, v_1^*) \in \mathbf{N}({\rm gph}(G), (g(x), \bar y)) \cap (\mathbf{B}_{\mathbb{E}^*} \times \mathbf{B}_{\mathbb{Y}^*})$ such that
\[
\frac{\tau \ell}{\tau + 1} \Big(\frac{1}{L_1 \tau} \nabla g(x)^*(e^*), y^*\Big) = \nabla \Psi(x, \bar y)(e_1^*, v_1^*) = (\nabla g(x)^*(e_1^*), v_1^*).
\]
This implies
\[
\nabla g(x)^*\Big(\frac{\ell}{(\tau + 1)L_1} e^*\Big) = \nabla g(x)^*(e_1^*) \quad \text{and} \quad \frac{\tau \ell}{\tau + 1} y^* = v_1^*.
\]
Since $\nabla g(x)^*$ is injective, it follows that
\[
\frac{\ell}{(\tau + 1)L_1} e^* = e_1^*, \quad \text{hence} \quad e^* = \frac{(\tau + 1)L_1}{\ell} e_1^* \in \frac{(\tau + 1)L_1}{\ell} D^*G(e, \bar y)(\mathbf{B}_{\mathbb{Y}^*}).
\]
Thus \eqref{3.38} holds with $\tau_1 := \frac{(\tau + 1)L_1}{\ell} > 0$.

Conversely, suppose that there exist $\tau_1, \delta_1 > 0$ such that \eqref{3.38} holds. Take $\delta \in (0, r)$ such that
\begin{equation}\label{3.40}
\|g(x) - g(\bar x)\| < \delta_1, \quad \forall x \in \mathbf{B}(\bar x, \delta).
\end{equation}
Let $x \in \mathbf{B}(\bar x, \delta) \cap F^{-1}(\bar y)$ and $x^* \in \widehat{\mathbf{N}}(F^{-1}(\bar y), x) \cap \mathbf{B}_{\mathbb{X}^*}$. Then $g(x) \in G^{-1}(\bar y) \cap \mathbf{B}(g(\bar x), \delta_1)$. By \eqref{3.25a}, there exists $e_1^* \in \mathbf{N}(G^{-1}(\bar y), g(x)) \cap \mathbf{B}_{\mathbb{E}^*}$ such that 
\begin{equation}\label{3.41}
\ell_1 x^* = \nabla g(x)^*(e_1^*).
\end{equation}
By \eqref{3.38}, there exists $y_1^* \in \mathbf{B}_{\mathbb{Y}^*}$ such that
\[
\Big(\frac{1}{\tau_1} e_1^*, y_1^*\Big) \in \mathbf{N}({\rm gph}(G), (g(x), \bar y)).
\]
It follows from \cite[Corollary 1.15]{Mordukhovich} that 
\[
\nabla \Psi(x, \bar y)^*\Big(\frac{1}{\tau_1} e_1^*, y_1^*\Big) \in \nabla \Psi(x, \bar y)^*(\mathbf{N}({\rm gph}(G), (g(x), \bar y))) = \widehat{\mathbf{N}}({\rm gph}(F), (x, \bar y)).
\]
Together with \eqref{3.41}, this implies
\[
\Big(\frac{\ell_1}{\tau_1} x^*, y_1^*\Big) = \nabla \Psi(x, \bar y)^*\Big(\frac{1}{\tau_1} e_1^*, y_1^*\Big) \in \widehat{\mathbf{N}}({\rm gph}(F), (x, \bar y)),
\]
and consequently
\[
x^* \in \frac{\tau_1}{\ell_1} \widehat D^*F(x, \bar y)(-y_1^*) \subseteq \frac{\tau_1}{\ell_1} \widehat D^*F(x, \bar y)(\mathbf{B}_{\mathbb{Y}^*}).
\]
Hence \eqref{3.23} holds with $\tau := \frac{\tau_1}{\ell_1} > 0$. The proof is complete. \end{proof}


\section{Applications to Metric Subregularity of Conic Inequalities}

As an application of the results obtained in Section 3, we study the metric subregularity of  {the conic inequality} defined by a proper vector-valued function between two Banach spaces and a closed cone with a nontrivial recession cone. We begin by formulating such a conic inequality.

Let $\mathbb{X}$ and $\mathbb{Y}$ be Banach spaces, and let $K \subseteq \mathbb{Y}$ be a nonempty closed cone with a nontrivial recession cone; that is,
\[
K^{\infty} \neq \{0\}.
\]
The cone $K$ induces a preorder $\leq_K$ on $\mathbb{Y}$ via
\begin{equation}\label{4.1}
  y_1 \leq_K y_2 
  \quad \Longleftrightarrow \quad
  y_2 - y_1 \in K.
\end{equation}

We denote by $\infty_{\mathbb{Y}}$ the abstract infinity.  
Let $\Lambda_K(\mathbb{X}, \mathbb{Y})$ be the set of all proper vector-valued functions 
\[
\varphi : \mathbb{X} \to \mathbb{Y} \cup \{\infty_{\mathbb{Y}}\}
\]
such that
\begin{equation}\label{4.2}
  {\rm epi}_K(\varphi) := 
  \big\{ (x, y) \in \mathbb{X} \times \mathbb{Y} : \varphi(x) \leq_K y \big\}
\end{equation}
is closed in the product space $\mathbb{X} \times \mathbb{Y}$.

Let $f \in \Lambda_K(\mathbb{X}, \mathbb{Y})$.  
Consider the following conic inequality:
\begin{equation}\label{4.3}
  f(x) \leq_K 0.
\end{equation}
We denote the solution set of \eqref{4.3} by
\[
\mathbf{S}_K(f) := \{ x \in \mathbb{X} : f(x) \leq_K 0 \}.
\]
Conic inequalities can encompass a wide variety of constraint systems in optimization.  
For example, let $\mathbb{Y} := \mathbb{R}^{m+n}$ and $K := \mathbb{R}^m_+ \times \{ 0_{\mathbb{R}^n} \}$.  
Then the conic inequality $f(x) \leq_K 0$ reduces to the usual finite system of inequalities and equalities.  
If $\mathbb{Y}$ is the space of all continuous real-valued functions defined on a compact topological space $T$,  
and $K$ is the cone of all nonnegative continuous functions on $T$,  
then such a conic inequality reduces to the constraint system of a semi-infinite optimization problem.

\medskip

Recall from \cite{Zheng2022-MOR, ZN2019-ESAIM} that $f$ is said to be \emph{metrically subregular} at  
$\bar{x} \in \mathbf{S}_K(f)$ if there exist $\tau, \delta \in (0, +\infty)$ such that
\begin{equation}\label{4.4}
  \mathbf{d}(x, \mathbf{S}_K(f)) \leq \tau \, \mathbf{d}\big(f(x), -K\big),
  \quad \forall x \in \mathbf{B}(\bar{x}, \delta).
\end{equation}

To study the metric subregularity of the vector-valued function $f$,  
we consider the following  {two types of subdifferential of $f$ with respect to $K$}.

\medskip

Let
\[
K^{\infty,+} := \big\{ y^* \in \mathbb{Y}^* : \langle y^*, y \rangle \geq 0, \ \forall y \in K^{\infty} \big\},
\]
and interpret $\langle y^*, \infty_{\mathbb{Y}} \rangle$ as $+\infty$ for all $y^* \in K^{\infty,+}$.

\medskip

Let $u \in {\rm dom}(f)$.  
We denote by $\widehat{\partial}_K f(u)$ and $\partial_K f(u)$ the Fr\'echet and limiting subdifferentials of $f$ at $u$ with respect to $K$, respectively, defined as
\begin{equation*}\label{4.5}
  \widehat{\partial}_K f(u)
  := \big\{ x^* \in \mathbb{X}^* :
     \widehat{\mathbf{N}}\big({\rm epi}_K(f), (u, f(u))\big)
     \cap \big( \{x^*\} \times \big(-K^{\infty,+} \cap \mathbf{S}_{\mathbb{Y}^*} \big) \big)
     \neq \emptyset \big\},
\end{equation*}
\begin{equation*}\label{4.5a}
  \partial_K f(u)
  := \big\{ x^* \in \mathbb{X}^* :
     \mathbf{N}\big({\rm epi}_K(f), (u, f(u))\big)
     \cap \big( \{x^*\} \times \big(-K^{\infty,+} \cap \mathbf{S}_{\mathbb{Y}^*} \big) \big)
     \neq \emptyset \big\}.
\end{equation*}

Similarly, we define the Fr\'echet and limiting \emph{singular} subdifferentials of $f$ at $u$ with respect to $K$ as
\begin{equation*}\label{4.6}
  \widehat{\partial}^{\infty}_K f(u)
  := \big\{ x^* \in \mathbb{X}^* :
     (x^*, 0) \in \widehat{\mathbf{N}}\big({\rm epi}_K(f), (u, f(u))\big) \big\},
\end{equation*}
\begin{equation*}\label{4.6a}
  \partial^{\infty}_K f(u)
  := \big\{ x^* \in \mathbb{X}^* :
     (x^*, 0) \in \mathbf{N}\big({\rm epi}_K(f), (u, f(u))\big) \big\}.
\end{equation*}

It is easy to verify that $\widehat{\partial}^{\infty}_K f(u)$ and $\partial^{\infty}_K f(u)$ are cones.  
Moreover, if $f$ is locally Lipschitz around $u$, then
\[
\widehat{\partial}^{\infty}_K f(u)
= \partial^{\infty}_K f(u)
= \{0\}.
\]
\begin{lemma}\label{lem4.1}
Let $f \in \Lambda_K(\mathbb{X}, \mathbb{Y})$ and define $F(x) := f(x) + K$ for all $x \in \mathbb{X}$.  
Then, for any $x \in {\rm dom}(f)$, one has ${\rm dom}\big(D^*F(x, f(x))\big) \subseteq K^{\infty,+}$,
\begin{equation}\label{4.5-250712}
D^*F(x, f(x))(\mathbb{Y}^*) \subseteq \mathbb{R}_+ \, \partial_K f(x) + \partial_K^{\infty} f(x),
\end{equation}
and
\begin{equation}\label{4.6-250712}
\widehat{D}^*F(x, f(x))(\mathbb{Y}^*) \subseteq \mathbb{R}_+ \, \widehat{\partial}_K f(x) + \widehat{\partial}_K^{\infty} f(x).
\end{equation}
\end{lemma}

\noindent
\begin{proof}
Let $x \in {\rm dom}(f)$.  
Take any $y^* \in {\rm dom}\big(D^*F(x, f(x))\big)$ and $x^* \in D^*F(x, f(x))(y^*)$.  
Then $(x^*, -y^*) \in \mathbf{N}({\rm gph}(F), (x, f(x)))$.  
By \eqref{2.1}, there exist sequences $\varepsilon_k \downarrow 0$, $(x_k, y_k) \xrightarrow{{\rm gph}(F)}(x, f(x))$,  
and $(x_k^*, -y_k^*) \xrightarrow{w^*} (x^*, -y^*)$ such that
\[
(x_k^*, -y_k^*) \in  {\widehat{\bf N}_{\varepsilon_k}}({\rm gph}(F), (x_k, y_k)) \quad \forall k.
\]
Thus, for each $k$, one has
\begin{equation}\label{4.8}
\limsup_{(u,v) \xrightarrow{{\rm gph}(F)} (x_k, y_k)}
\frac{\langle x_k^*, u - x_k \rangle - \langle y_k^*, v - y_k \rangle}
     {\|(u - x_k, v - y_k)\|}
\leq \varepsilon_k.
\end{equation}

Let $e \in K^{\infty}$.  
For any sufficiently small $\epsilon > 0$, we have
\[
y_k + \epsilon e - f(x_k) \in K + \epsilon e \subseteq K
\]
(since $e \in K^{\infty}$).  
This implies $(x_k, y_k + \epsilon e)  {\xrightarrow{{\rm gph}(F)}} (x_k, y_k)$ as $\epsilon \to 0^+$,  
and hence, by \eqref{4.8},
\[
- \langle y_k^*, e \rangle \leq \varepsilon_k.
\]
Taking the limit as $k \to \infty$ gives $\langle y^*, e \rangle \geq 0$, so $y^* \in K^{\infty,+}$.

If $y^* \neq 0$, then
\[
\left( \frac{x^*}{\|y^*\|}, -\frac{y^*}{\|y^*\|} \right) \in \mathbf{N}({\rm gph}(F), (x, f(x))),
\]
which implies $\frac{x^*}{\|y^*\|} \in \partial_K f(x)$, since ${\rm gph}(F) = {\rm epi}_K(f)$.  
Thus
\[
x^* \in \|y^*\| \, \partial_K f(x) \subseteq \mathbb{R}_+ \, \partial_K f(x).
\]

If $y^* = 0$, then $(x^*, 0) \in \mathbf{N}({\rm gph}(F), (x, f(x)))$,  
so $x^* \in \partial_K^{\infty} f(x)$.  
This proves \eqref{4.5-250712}.

The proof of \eqref{4.6-250712} is analogous and is omitted.  
\end{proof}

\begin{remark}
It is straightforward to verify that if $\mathbb{Y} := \mathbb{R}$ and $K := [0, +\infty)$,  
then the Fr\'echet and limiting subdifferentials with respect to $K$ reduce to the usual ones.  

Furthermore, if we define $F(x) := f(x) + K$ for any $x \in \mathbb{X}$, then one can check that
\[
\widehat{\partial}_K f(u) = \widehat{D}^*F(x, f(x))(\mathbf{S}_{\mathbb{Y}^*}), 
\quad
\partial_K f(u) = D^*F(x, f(x))(\mathbf{S}_{\mathbb{Y}^*}),
\]
and
\[
\widehat{\partial}_K^{\infty} f(u) = \widehat{D}^*F(x, f(x))(0),  
\quad
\partial_K^{\infty} f(u) = D^*F(x, f(x))(0).
\]
\end{remark}

\medskip

The next theorem provides necessary conditions for the metric subregularity of the conic inequality,  
and characterizes Asplund spaces in terms of such conditions.

\begin{theorem}\label{th4.1}
Let $\mathbb{X}$ be a Banach space. The following statements are equivalent:
\begin{itemize}
\item[\rm (i)] $\mathbb{X}$ is an Asplund space.

\item[\rm (ii)] For any Asplund space $\mathbb{Y}$, any nonempty closed cone $K \subseteq \mathbb{Y}$ with a nontrivial recession cone, and any $f \in \Lambda_K(\mathbb{X}, \mathbb{Y})$  {being} metrically subregular  {at $\bar{x} \in \mathbf{S}_K(f)$} with $f(\bar{x}) = 0$,  
there exists $\delta > 0$ such that
\begin{equation}\label{4.9}
\mathbf{N}(\mathbf{S}_K(f), x) \subseteq \mathbb{R}_+ \, \partial_K f(x) + \partial_K^{\infty} f(x)
\end{equation}
 {holds} for all $x \in \mathbf{S}_K(f) \cap \mathbf{B}(\bar{x}, \delta)$ with $f(x) = 0$. 

\item[\rm (iii)] For any Asplund space $\mathbb{Y}$, any nonempty closed cone $K \subseteq \mathbb{Y}$ with a nontrivial recession cone, and any $f \in \Lambda_K(\mathbb{X}, \mathbb{Y})$ that is metrically subregular at $\bar{x} \in \mathbf{S}_K(f)$ with $f(\bar{x}) = 0$,  
relation \eqref{4.9} holds for $x = \bar{x}$.

\item[\rm (iv)] For any Asplund space $\mathbb{Y}$, any nonempty closed cone $K \subseteq \mathbb{Y}$ with a nontrivial recession cone, and any continuous vector-valued function $f: \mathbb{X} \to \mathbb{Y}$  {being} metrically subregular at $\bar{x} \in \mathbf{S}_K(f)$ with $f(\bar{x}) = 0$,  
relation \eqref{4.9} holds for $x = \bar{x}$.
\end{itemize}
\end{theorem}
\begin{proof}
(i) $\,\Rightarrow\,$ (ii):  
Let $\mathbb{Y}$ be an Asplund space, let $K \subseteq \mathbb{Y}$ be a nonempty closed cone with a nontrivial recession cone, and let $f \in \Lambda_K(\mathbb{X}, \mathbb{Y})$ be metrically subregular at some $\bar{x} \in \mathbf{S}_K(f)$ with $f(\bar{x}) = 0$.  
Define $F(x) := f(x) + K$ for any $x \in \mathbb{X}$. Then $F \in \Gamma(\mathbb{X}, \mathbb{Y})$, and one can verify that $F$ is metrically subregular at $(\bar{x}, 0)$ if and only if the conic inequality $f(x) \leq_K 0$ is metrically subregular at $\bar{x}$.  

By \cref{th3.1}, there exists $\delta > 0$ such that for any $x \in \mathbf{B}(\bar{x}, \delta) \cap F^{-1}(0)$, we have
\begin{equation}\label{4.8-250620}
    \mathbf{N}(F^{-1}(0), x) \subseteq D^*F(x, 0)(\mathbb{Y}^*).
\end{equation}
Now, for any $x \in \mathbf{S}_K(f) \cap \mathbf{B}(\bar{x}, \delta)$ with $f(x) = 0$,  
\cref{lem4.1} yields
\[
D^*F(x, 0)(\mathbb{Y}^*) \subseteq \mathbb{R}_+ \, \partial_K f(x) + \partial_K^{\infty} f(x),
\]
and combining this with \eqref{4.8-250620} gives
\[
\mathbf{N}(F^{-1}(0), x) \subseteq \mathbb{R}_+ \, \partial_K f(x) + \partial_K^{\infty} f(x).
\]
Since $F^{-1}(0) = \mathbf{S}_K(f)$, this proves (ii).  

The implications (ii) $\,\Rightarrow\,$ (iii) $\,\Rightarrow\,$ (iv) are immediate, so it remains to prove that (iv) $\,\Rightarrow\,$ (i).

Suppose, for contradiction, that $\mathbb{X}$ is not an Asplund space.  
Following \cite{MW2000,FM1998}, we can write $\mathbb{X} = \mathbb{Z} \times \mathbb{R}$ with the norm
\[
\|(z, \alpha)\| := \|z\| + |\alpha|, \quad (z, \alpha) \in \mathbb{Z} \times \mathbb{R},
\]
where $\mathbb{Z}$ is not an Asplund space.  
By \cite[Theorem~1.5.3]{DGZ1993} (see also \cite[Theorem~2.1]{FM1998}), there exists an equivalent norm $|||\cdot|||$ on $\mathbb{Z}$ and a constant $\gamma > 0$ such that \eqref{3.10} holds.

Define $\varphi : \mathbb{Z} \to \mathbb{R}$ by $\varphi(z) := -|||z|||$, and set
\[
A_1 := \{0_{\mathbb{Z}}\} \times (-\infty, 0], 
\quad 
A_2 := \operatorname{epi}(\varphi), 
\quad 
\bar{x} := 0_{\mathbb{X}}, 
\quad 
\bar{y} := (\bar{x}, \bar{x}).
\]
Let $\mathbb{Y} := \mathbb{X}^2$ be equipped with the $\ell^1$-norm, and define
\[
K := -(A_1 \times A_2).
\]
Then $K \subseteq \mathbb{Y}$ is a nonempty closed cone with a nontrivial recession cone (indeed, $K^{\infty} \supseteq -A_1 \times \{0\}$).  
Define $f : \mathbb{X} \to \mathbb{Y}$ by $f(x) := (x, x)$, and consider the conic inequality
\[
f(x) \leq_K 0.
\]
Let $F(x) := f(x) + K$ for all $x \in \mathbb{X}$. Then $\operatorname{epi}_K(f) = \operatorname{gph}(F)$.  
\medskip 
We first claim that
\begin{equation}\label{4.11}
    \partial_K f(\bar{x}) \cup \partial_K^{\infty} f(\bar{x}) 
    \subseteq \mathbf{N}(A_1, \bar{x}) + \mathbf{N}(A_2, \bar{x}).
\end{equation}
Indeed, let $x^* \in \partial_K f(\bar{x}) \cup \partial_K^{\infty} f(\bar{x})$.  
Then there exists $y^* = \big(y^*(1), y^*(2)\big) \in K^{\infty,+}$ such that
\[
(x^*, -y^*) \in \mathbf{N}(\operatorname{epi}_K(f), (\bar{x}, f(\bar{x}))) 
= \mathbf{N}(\operatorname{gph}(F), (\bar{x}, f(\bar{x}))),
\]
and by \cref{lem3.1},
\[
y^* \in \mathbf{N}(A_1, \bar{x}) \times \mathbf{N}(A_2, \bar{x}),
\quad
x^* = y^*(1) + y^*(2).
\]
This proves \eqref{4.11}.

From the proof of \eqref{3.11}, we also have
\[
\mathbf{d}(x, \mathbf{S}_K(f)) \leq 2 \, \mathbf{d}(f(x), -K),
\quad \forall x = (z, \alpha) \in \mathbb{Z} \times \mathbb{R},
\]
which means that the conic inequality $f(x) \leq_K 0$ is metrically subregular at $\bar{x} \in \mathbf{S}_K(f)$.  

By \eqref{3.17-250710} and \eqref{3.18-250710}, we have
\begin{equation}\label{4.12}
    \mathbf{N}(A_1, \bar{x}) = \mathbb{Z}^* \times [0, +\infty), 
    \quad
    \mathbf{N}(A_2, \bar{x}) = \{(0, 0)\},
\end{equation}
and
\begin{equation}\label{4.13}
    \mathbf{N}(\mathbf{S}_K(f), \bar{x}) 
    = \mathbf{N}(F^{-1}(\bar{y}), \bar{x}) 
    = \mathbb{Z}^* \times \mathbb{R} 
    \not\subseteq \mathbf{N}(A_1, \bar{x}) + \mathbf{N}(A_2, \bar{x}).
\end{equation}

On the other hand, \eqref{4.11} and \eqref{4.12} imply
\[
\mathbb{R}_+ \, \partial_K f(\bar{x}) + \partial_K^{\infty} f(\bar{x})
\subseteq \mathbf{N}(A_1, \bar{x}) + \mathbf{N}(A_2, \bar{x}),
\]
which, together with \eqref{4.13}, gives
\[
\mathbf{N}(\mathbf{S}_K(f), \bar{x}) 
\not\subseteq \mathbb{R}_+ \, \partial_K f(\bar{x}) + \partial_K^{\infty} f(\bar{x}),
\]
contradicting \eqref{4.9} with $x = \bar{x}$.  
Thus $\mathbb{X}$ must be Asplund.  
\end{proof}

\begin{remark}
It is worth noting that in the special case where $\mathbb{Y} := \mathbb{R}$ and $K := [0,+\infty)$, metric subregularity of the conic inequality reduces to the local error bound for the scalar inequality $f(x) \leq 0$ in Asplund spaces. This was the subject of our recent work in \cite{WTY2025-JCA}, and Theorem~\ref{th4.1} actually recovers \cite[Theorem~3.1]{WTY2025-JCA}. 

Moreover, it is well known that Lewis and Pang \cite[Proposition~2]{LewisPang1998} studied error bounds for convex inequality systems and established necessary conditions for such bounds in terms of normal cones to the solution set and subdifferentials of the underlying convex function. In particular, their result shows that a local error bound implies the basic constraint qualification (BCQ). More precisely:

\emph{Given a lower semicontinuous convex function $f : \mathbb{R}^m \to \mathbb{R} \cup \{+\infty\}$, if the convex inequality $f(x) \leq 0$ admits a local error bound at $\bar{x} \in  { \mathbf{S}_f} := \{ x \in \mathbb{R}^m : f(x) \leq 0 \}$ with $f(\bar{x}) = 0$, then the following BCQ holds at $\bar{x}$:}
\[
\mathbf{N}( { \mathbf{S}_f},\bar{x}) \subseteq \mathbb{R}_+ \, \partial f(\bar{x}).
\]
When restricted to convex inequalities, Theorem~\ref{th4.1}(ii) recovers \cite[Proposition~2]{LewisPang1998}, since
\[
\mathbb{R}_+ \, \partial f(\bar{x}) + \partial^{\infty} f(\bar{x}) \subseteq \mathbb{R}_+ \, \partial f(\bar{x}).
\]
From the viewpoint of Theorem~\ref{th4.1}, the validity of such implication results in both \cite[Theorem~3.1]{WTY2025-JCA} and \cite[Proposition~2]{LewisPang1998} essentially stems from the Asplund property of the underlying space.
\end{remark}

 {It is known that for a given closed cone and a vector-valued function, the corresponding conic inequality can be equivalently described by a special multifunction, while the reverse case seems not true necessarily. However, based on \cref{th3.1} and \cref{th4.1}, the following theorem  establishes an equivalence result between conic inequalities and multifunctions in the sense of proving necessary dual conditions for metric subregularity in terms of limiting normal cones; that is,}

 {\begin{theorem}
		Let $\mathbb{X}$ be a Banach space. Then following statements are equivalent:
		\begin{itemize}
			\item[\rm (i)] For any Asplund space $\mathbb{Y}$, any nonempty closed cone $K \subseteq \mathbb{Y}$ with a nontrivial recession cone, and any $f \in \Lambda_K(\mathbb{X}, \mathbb{Y})$ being metrically subregular at  $\bar{x} \in \mathbf{S}_K(f)$ with $f(\bar{x}) = 0$,  one has
			\begin{equation}
				\mathbf{N}(\mathbf{S}_K(f), \bar x) \subseteq \mathbb{R}_+ \, \partial_K f(\bar x) + \partial_K^{\infty} f(\bar x)
			\end{equation}
			\item[\rm (ii)] For every Asplund space $\mathbb{Y}$ and every $F \in \Gamma(\mathbb{X}, \mathbb{Y})$ being metrically subregular at $(\bar x, \bar y) \in {\rm gph}(F)$, $F$ satisfies the limiting $\mathbf{BCQ}$ at $(\bar x, \bar y)$.
		\end{itemize}
\end{theorem}}

Next, we establish necessary or sufficient conditions for a Banach space to be Asplund in terms of  {Fr\'echet  subdifferentials} of vector-valued functions with respect to $K$. For this purpose, we first recall a lemma on the Fr\'echet subdifferential of a given vector-valued function with respect to $K$.

\begin{lemma}\label{lem4.2}
Let $\mathbb{X}$ and $\mathbb{Y}$ be Banach spaces, and let $K \subseteq \mathbb{Y}$ be a nonempty closed convex cone. Suppose $f \in \Lambda_K(\mathbb{X},\mathbb{Y})$ and define $F(x) := f(x) + K$ for all $x \in \mathbb{X}$. Then, for any $(u,v) \in \mathrm{gph}(F)$, we have
\begin{equation}\label{4.14}
\widehat{\mathbf{N}}(\mathrm{gph}(F), (u,v)) \subseteq \widehat{\mathbf{N}}(\mathrm{epi}_K(f), (u,f(u))).
\end{equation}
Furthermore, if $f$ is continuous at $u$, then
\begin{equation}\label{4.15}
\mathbf{N}(\mathrm{gph}(F), (u,v)) \subseteq \mathbf{N}(\mathrm{epi}_K(f), (u,f(u))).
\end{equation}
\end{lemma}

\begin{proof}
Let $(u^*,v^*) \in \widehat{\mathbf{N}}(\mathrm{gph}(F),(u,v))$. Then
\begin{equation}\label{4.16}
\limsup_{(x,y) \xrightarrow{\mathrm{gph}(F)} (u,v)} \frac{\langle u^*,x-u\rangle + \langle v^*,y-v\rangle}{\|(x-u,y-v)\|} \leq 0.
\end{equation}
For any $(x,y) \to (u,f(u))$ with $(x,y) \in \mathrm{epi}_K(f)$, we have
\[
y + v - f(u) \to v, \quad\text{and}\quad y + v - f(u) \in f(x) + K + K \subseteq f(x) + K,
\]
since $K$ is a convex cone. Thus, by \eqref{4.16},
\[
\limsup_{(x,y) \xrightarrow{\mathrm{epi}_K(f)} (u,f(u))} \frac{\langle u^*,x-u\rangle + \langle v^*,y-f(u)\rangle}{\|(x-u,y-f(u))\|} \leq 0,
\]
which shows that $(u^*,v^*) \in \widehat{\mathbf{N}}(\mathrm{epi}_K(f),(u,f(u)))$.

Now assume $f$ is continuous at $u$. Let $(u^*,v^*) \in \mathbf{N}(\mathrm{gph}(F),(u,v))$. Then there exist sequences $\epsilon_k \downarrow 0$, $(u_k,v_k) \xrightarrow{\mathrm{gph}(F)} (u,v)$, and $(u_k^*,v_k^*) \xrightarrow{w^*} (u^*,v^*)$ such that
\[
(u_k^*,v_k^*) \in  {\widehat{\bf N}_{\epsilon_k}}(\mathrm{gph}(F),(u_k,v_k)), \quad \forall k.
\]
That is,
\begin{equation}\label{4.17}
\limsup_{(x,y) \xrightarrow{\mathrm{gph}(F)} (u_k,v_k)} \frac{\langle u_k^*,x-u_k\rangle + \langle v_k^*,y-v_k\rangle}{\|(x-u_k,y-v_k)\|} \leq \epsilon_k.
\end{equation}
For any $(x,y) \to (u_k,f(u_k))$ with $(x,y) \in \mathrm{epi}_K(f)$, we have
\[
y + v_k - f(u_k) \to v_k, \quad\text{and}\quad y + v_k - f(u_k) \in f(x) + K + K \subseteq f(x) + K.
\]
Thus, from \eqref{4.17},
\[
\limsup_{(x,y) \xrightarrow{\mathrm{epi}_K(f)} (u_k,f(u_k))} \frac{\langle u_k^*,x-u_k\rangle + \langle v_k^*,y-f(u_k)\rangle}{\|(x-u_k,y-f(u_k))\|} \leq \epsilon_k,
\]
which means
\[
(u_k^*,v_k^*) \in  {\widehat{\bf N}_{\epsilon_k}}(\mathrm{epi}_K(f),(u_k,f(u_k))).
\]
Since $u_k \to u$ and $f(u_k) \to f(u)$ by continuity, passing to the limit yields
\[
(u^*,v^*) \in \mathbf{N}(\mathrm{epi}_K(f),(u,f(u))).
\]
Thus \eqref{4.15} holds. 
\end{proof}

In terms of Fr\'echet subdifferentials of the given vector-valued function with respect to $K$, the following theorem provides necessary or sufficient conditions for a Banach space to be Asplund.


\begin{theorem}\label{theorem4.2}
Let $\mathbb{X}$ be a Banach space. Consider the following statements:
\begin{itemize}
    \item[\rm (i)] Suppose that $\mathbb{X}$ is an Asplund space. Then for any Asplund space $\mathbb{Y}$, any nonempty closed convex cone $K\subseteq \mathbb{Y}$, and any $f\in\Lambda_K(\mathbb{X},\mathbb{Y})$  {being} metrically subregular  {at $\bar x\in \mathbf{S}_K(f)$}, there exists $\delta>0$ such that for any $\varepsilon>0$,  {one has}
    \begin{equation}\label{4.9a}
        \widehat{\mathbf{N}}(\mathbf{S}_K(f), x) \subseteq  {\bigcup \left\{\mathbb{R}_+\widehat\partial_K f(u) + \widehat\partial_K^{\infty} f(u): u\in \mathbf{B}(x,\varepsilon)\right\}} + \varepsilon\, {\bf B}_{\mathbb{X}^*}
    \end{equation}
     {holds} for all $x\in \mathbf{S}_K(f) \cap \mathbf{B}(\bar x,\delta)$.

    \item[\rm (ii)] Suppose that for any Asplund space $\mathbb{Y}$, any nonempty closed cone $K\subseteq \mathbb{Y}$ with a nontrivial recession cone, and any $f\in\Lambda_K(\mathbb{X},\mathbb{Y})$  {being} metrically subregular  {at $\bar x\in \mathbf{S}_K(f)$}, one has
    \begin{equation}\label{4.9b}
        \widehat{\mathbf{N}}(\mathbf{S}_K(f), \bar x) \subseteq  {\bigcup \left\{\mathbb{R}_+\widehat\partial_K f(u) + \widehat\partial_K^{\infty} f(u): u\in \mathbf{B}(\bar x,\varepsilon)\right\}} 
        + \varepsilon\, {\bf B}_{\mathbb{X}^*}
    \end{equation}
    holds for any $\varepsilon>0$. Then $\mathbb{X}$ is an Asplund space.
\end{itemize}
\end{theorem}
\begin{proof}
(i) Let $\mathbb{Y}$ be an Asplund space, $K\subseteq \mathbb{Y}$ a nonempty closed convex cone, and let $f\in\Lambda_K(\mathbb{X},\mathbb{Y})$ be metrically subregular at $\bar x\in \mathbf{S}_K(f)$.  
Define $F:\mathbb{X}\to\mathbb{Y}$ by
\[
F(x):=f(x)+K, \quad \forall x\in \mathbb{X}.
\]
Then $F\in\Gamma(\mathbb{X},\mathbb{Y})$ and $F^{-1}(0)=\mathbf{S}_K(f)$.  
Moreover, the conic inequality $f(x)\leq_K 0$ is metrically subregular at $\bar x$ if and only if $F$ is metrically subregular at $(\bar x,0)\in{\rm gph}(F)$.  

By \cref{theorem3.2}, there exists $\delta>0$ such that for any  
$x\in F^{-1}(0)\cap {\bf B}(\bar x,\delta)$ and any $\epsilon>0$,
\begin{equation}\label{4.19a}
	\widehat{\mathbf{N}}\big(F^{-1}(0), x\big)
	\subseteq \bigcup_{(u,v)\in \mathbf{B}((\bar x,0),\epsilon)\cap {\rm gph}(F)}
	\widehat D^*F(u,v)(\mathbb{Y}^*) + \epsilon\,\mathbf{B}_{\mathbb{X}^*}.
\end{equation}

We claim that
\begin{equation}\label{4.20a}
	\widehat D^*F(u,v)(\mathbb{Y}^*) \subseteq 
	\mathbb{R}_+\,\widehat\partial_K f(u) + \widehat\partial_K^{\infty} f(u),
	\quad \forall (u,v)\in {\rm gph}(F).
\end{equation}
Assuming this, \eqref{4.19a} immediately yields \eqref{4.9a}.

To verify \eqref{4.20a}, let $(u,v)\in {\rm gph}(F)$ and take $x^*\in \widehat D^*F(u,v)(\mathbb{Y}^*)$.  
Then there exists $y^*\in \mathbb{Y}^*$ such that
\[
(x^*,-y^*)\in\widehat{\mathbf{N}}({\rm gph}(F), (u,v)).
\]
By \cref{lem4.2},
\begin{equation}\label{4.21a}
	(x^*,-y^*)\in\widehat{\mathbf{N}}({\rm epi}_{K}(f), (u, f(u))),
\end{equation}
and hence \cref{lem4.1} implies
\[
x^*\in \widehat D^*F(u,f(u))(y^*)
\subseteq \mathbb{R}_+\,\widehat\partial_K f(u) + \widehat\partial_K^{\infty} f(u),
\]
which establishes \eqref{4.20a}.

\smallskip
(ii) Suppose,  {on} the contrary, that $\mathbb{X}$ is not an Asplund space.  
By \cite{MW2000,FM1998}, $\mathbb{X}$ can be written as $\mathbb{X}=\mathbb{Z}\times \mathbb{R}$ endowed with the norm  
$\|(z,\alpha)\|:=\|z\|+|\alpha|$ for $(z,\alpha)\in \mathbb{X}$, where $\mathbb{Z}$ is not Asplund.  
By \cite[Theorem 1.5.3]{DGZ1993} (see also \cite[Theorem 2.1]{FM1998}), there exists an equivalent norm $|||\cdot|||$ on $\mathbb{Z}$ and $\gamma>0$ such that \eqref{3.10} holds.

Define $\varphi:\mathbb{Z}\to\mathbb{R}$ by $\varphi(z):=-|||z|||$, and set
\[
A_1:=\{0_{\mathbb{Z}}\}\times (-\infty,0], 
\quad A_2:=\epi(\varphi),
\quad \bar x:=0_\mathbb{X},
\quad \bar y:=(\bar x,\bar x).
\]
Let $\mathbb{Y}:=\mathbb{X}^2$ with the $\ell^1$-norm, and define  
$K:=-(A_1\times A_2)$.  
Then $K$ is a nonempty closed cone with a nontrivial recession cone (since $K^{\infty}\supseteq -A_1\times \{0\}$).  

Define $f:\mathbb{X}\to\mathbb{Y}$ by $f(x):=(x,x)$.  
The conic inequality
\[
f(x)\leq_K 0
\]
has solution set $\mathbf{S}_K(f)$, and with $F(x):=f(x)+K$ we have ${\rm epi}_K(f)={\rm gph}(F)$.

We will show that for any $\varepsilon>0$,
\begin{equation}\label{4.9c}
	\widehat{\mathbf{N}}(\mathbf{S}_K(f), \bar x)
	\not\subseteq 
	\bigcup_{u\in \mathbf{B}(\bar x,\varepsilon)}\left[\mathbb{R}_+\,\widehat\partial_K f(u) + \widehat\partial_K^{\infty} f(u)\right]
	+ \varepsilon\,{\bf B}_{\mathbb{X}^*},
\end{equation}
contradicting \eqref{4.9b}.

Let $\varepsilon>0$ and $u\in \mathbf{B}(\bar x,\varepsilon)$.  
We claim that
\begin{equation}\label{4.11a}
	\widehat\partial_K f(u) \cup \widehat\partial_K^{\infty} f(u) 
	\subseteq \widehat{\mathbf{N}}(A_1,0) + \widehat{\mathbf{N}}(A_2,0).
\end{equation}
Indeed, take $x^*\in \widehat\partial_K f(u) \cup \widehat\partial_K^{\infty} f(u)$.  
Then there exists $y^*=(y^*(1),y^*(2))\in K^{\infty,+}$ such that
\[
(x^*,-y^*) \in \widehat{\mathbf{N}}({\rm epi}_K(f),(u,f(u)))
= \widehat{\mathbf{N}}({\rm gph}(F),(u,f(u))).
\]
By \cref{lem3.1},
\[
y^* \in \widehat{\mathbf{N}}(A_1,0)\times \widehat{\mathbf{N}}(A_2,0),
\quad
x^* = y^*(1) + y^*(2),
\]
since $f(u)=(u,u)$, which proves \eqref{4.11a}.

From the proof of \eqref{3.11} we have
\[
\mathbf{d}(x, \mathbf{S}_K(f)) \leq 2\,\mathbf{d}(f(x), -K)
\]
for all $x=(z,\alpha)\in \mathbb{Z}\times \mathbb{R}$, and so $f(x)\leq_K 0$ is metrically subregular at $\bar x$.

By \eqref{2} and \eqref{3},  {one has}
\[
\widehat{\mathbf{N}}(A_1,0) = \mathbb{Z}^*\times [0,+\infty),
\quad
\widehat{\mathbf{N}}(A_2,0) = \{(0,0)\}.
\]
Moreover,
\[
\widehat{\mathbf{N}}(\mathbf{S}_K(f),\bar x)
= \widehat{\mathbf{N}}(F^{-1}(\bar y),\bar x)
= \mathbb{Z}^*\times \mathbb{R},
\quad
\varepsilon\,\mathbf{B}_{\mathbb{X}^*} \subseteq \mathbb{Z}^* \times [-\varepsilon,\varepsilon].
\]
Thus,
\[
\widehat{\mathbf{N}}(F^{-1}(\bar y),\bar x)
\not\subseteq 
\mathbb{Z}^*\times [0,+\infty) + \{(0,0)\} + \varepsilon\,\mathbf{B}_{\mathbb{X}^*}.
\]
Combining this with \eqref{4.11a} yields \eqref{4.9c}, a contradiction.  
The proof is complete.\end{proof}

\medskip

\begin{remark}\label{remark4.1}
It is noted that the conclusion in (i) of \cref{theorem4.2} may be invalid if the convexity assumption of $K\subseteq \mathbb{Y}$ is dropped. For example, let $\mathbb{X}:=\mathbb{R}$, $\mathbb{Y}:=\mathbb{R}^2$ and $K:={\rm epi}(\varphi)$ where $\varphi(x):=-|x|$, $x\in\mathbb{X}$. Then $K\subseteq \mathbb{Y}$ is a non-convex and closed cone with a nontrivial recession cone. Define $f:\mathbb{X}\rightarrow \mathbb{Y}$ by $f(x):=(x,2x)$ for all $x\in\mathbb{X}$ and let $\bar x:=0$. We consider the metric subregularity of the conic inequality $f(x)\leq_K 0$ at $\bar x$. It is easy to verify that $\mathbf{S}_K(f)=(-\infty, 0]$ and thus $\widehat{\mathbf{N}}(\mathbf{S}_K(f),\bar x)=[0,+\infty)$.

Let $\varepsilon>0$. We claim that for any $u\in \mathbf{B}(\bar x,\varepsilon)$, one has
\[
 {\widehat{\mathbf{N}}({\rm epi}_K(f),(u,f(u)))\subseteq \{0\}\times \mathbb{Y}^*}.
\]
Granting this, it follows that 
\[
\widehat{\mathbf{N}}(\mathbf{S}_K(f), \bar x) \not\subseteq 
\bigcup\left\{\mathbb{R}_+\widehat\partial_K f(u)+\widehat \partial_K^{\infty}f(u) : u\in \mathbf{B}(\bar x,\varepsilon)\right\} + \varepsilon {\bf B}_{\mathbb{X}^*},
\]
which implies that the conclusion in (i) of \cref{theorem4.2} does not hold.

Let $u\in \mathbf{B}(\bar x,\varepsilon)$ and $(u^*,-v^*)\in\widehat{\mathbf{N}}({\rm epi}_K(f),(u,f(u)))$ with $v^*=(v^*(1),v^*(2))$. Then
\begin{equation}\label{4-23a}
\limsup_{(x,y) \xrightarrow{{\rm epi}_K(f)}(u,f(u))} 
\frac{\langle u^*,x-u\rangle - \langle v^*,y-f(u)\rangle}{\|(x-u, y-f(u))\|} \leq 0.
\end{equation}

For any $x\rightarrow u$ with $x<u$ and $y:=f(u)$, one has $(x,y)\rightarrow(u,f(u))$ with $(x,y)\in {\rm epi}_K(f)$, and it follows from \eqref{4-23a} that $u^*\geq 0$.

For any $x\rightarrow u$ with $x>u$ and $y:=f(x)$, one has $(x,y)\rightarrow(u,f(u))$ with $(x,y)\in {\rm epi}_K(f)$, and \eqref{4-23a} implies
\begin{equation}\label{4-24a}
u^* - (v^*(1)+2v^*(2)) \leq 0.
\end{equation}

For sufficiently small $\varepsilon$, let $x:=u$ and $y:=f(u)+(\varepsilon, 0)$. Then $(x,y)\xrightarrow{{\rm epi}_K(f)}(u,f(u))$ as $\varepsilon\rightarrow 0$, and \eqref{4-23a} gives $v^*(1) = 0$.

For sufficiently small $\varepsilon$, let $x:=u$ and $y:=f(u)+(2\varepsilon, \varepsilon)$. Then $(x,y)\xrightarrow{{\rm epi}_K(f)}(u,f(u))$ as $\varepsilon\rightarrow 0$, and \eqref{4-23a} gives $v^*(2) = 0$. This and \eqref{4-24a} imply $u^* \leq 0$, and consequently $u^*=0$ (thanks to $u^*\geq 0$). \hfill $\Box$
\end{remark}

\begin{remark}
It is known from \cite{Mordukhovich} that limiting and Fr\'echet normal cones can be used to characterize Asplund spaces.  
In particular, \cite[Theorems~2.22 and~20]{Mordukhovich} show that the exact and approximate extremal principles, formulated via limiting and Fr\'echet normal cones, provide characterizations of Asplund spaces.  
Moreover, \cite[Theorem~2.30]{Mordukhovich} establishes exact sum rules for limiting subdifferentials and a fuzzy sum rule for Fr\'echet subdifferentials in the semi-Lipschitzian setting; each of these rules, when applied to all semi-Lipschitzian sums, is again shown to characterize Asplund spaces.  
Compared with these classical results, \cref{th4.1} and \cref{theorem4.2} reveal that the limiting-type subdifferentials (with respect to the cone) differ slightly from their Fr\'echet counterparts in how they characterize Asplund spaces.
\end{remark}

 {Based on \cref{theorem3.2} and \cref{theorem4.2}, the next theorem reveals the close connection between coinc inequalities and multifunctions when dealing with necessary dual conditions for metric subregularity in terms of Fr\'echet normal cones; that is,}

 {	\begin{theorem}
		Let $\mathbb{X}$ be a Banach space. Consider the following statements:
		\begin{itemize}
			\item[\rm (i)] For any Asplund space $\mathbb{Y}$, any nonempty closed cone $K\subseteq \mathbb{Y}$ with a nontrivial recession cone, and any $f\in\Lambda_K(\mathbb{X},\mathbb{Y})$ being metrically subregular  {at $\bar x\in \mathbf{S}_K(f)$}, one has \eqref{4.9b} holds for all $\varepsilon>0$. 
						\item [\rm(ii)] For every Asplund space $\mathbb{Y}$ and every $F\in\Gamma(\mathbb{X}, \mathbb{Y})$ being metrically subregular at $(\bar x,\bar y)\in{\rm gph}(F)$,  one has \eqref{3-16a} holds for all $\varepsilon>0$.
				   \item[\rm (iii)] For any Asplund space $\mathbb{Y}$, any nonempty closed convex cone $K\subseteq \mathbb{Y}$, and any $f\in\Lambda_K(\mathbb{X},\mathbb{Y})$ being metrically subregular at $\bar x\in \mathbf{S}_K(f)$, there exists $\delta>0$ such that for any $\varepsilon>0$, one has
				\eqref{4.9a}	holds for all $x\in \mathbf{S}_K(f) \cap \mathbf{B}(\bar x,\delta)$.
		\end{itemize}
		Then $\rm(i)\Rightarrow \rm(ii)\Rightarrow\rm(iii)$.
\end{theorem}}

\medskip

Using Fr\'echet and singular subdifferentials of vector-valued functions, we now present a sharper necessary condition for the metric subregularity of conic inequalities.

\begin{theorem}
Let $\mathbb{X}$ and $\mathbb{Y}$ be Asplund spaces, let $K\subseteq \mathbb{Y}$ be a nonempty closed convex cone, and let $f\in\Lambda_K(\mathbb{X},\mathbb{Y})$ be such that the conic inequality $f(x)\leq_K 0$ is metrically subregular at some $\bar x\in \mathbf{S}_K(f)$.  
Then there exist constants $\tau,\delta\in(0,+\infty)$ such that for every $\epsilon>0$,  {one has}
\begin{equation}\label{4.18}
\widehat{\mathbf{N}}(\mathbf{S}_K(f), x)\cap {\bf B}_{\mathbb{X}^*}
\ \subseteq\  {\bigcup}
\left\{
[0,(1+\epsilon)\tau]\,\widehat\partial_K f(u) + \widehat\partial_K^{\infty} f(u) : u\in \mathbf{B}(x,\epsilon)
\right\}
+ \epsilon\, {\bf B}_{\mathbb{X}^*}
\end{equation}
holds for all $x\in \mathbf{S}_K(f)\cap \mathbf{B}(\bar x,\delta)$.
\end{theorem}

\begin{proof}
Define 
\[
F(x) := f(x) + K, \quad \forall x \in \mathbb{X}.
\] 
Then $F \in \Gamma(\mathbb{X},\mathbb{Y})$ and $F^{-1}(0) = \mathbf{S}_K(f)$.  
Moreover, one can verify that the conic inequality $f(x) \leq_K 0$ is metrically subregular at $\bar{x}$ if and only if $F$ is metrically subregular at $(\bar{x},0) \in \mathrm{gph}(F)$.  

By virtue of \cref{th3.2}, there exist $\tau, \delta > 0$ such that for any 
$x \in F^{-1}(0) \cap \mathbf{B}(\bar{x},\delta)$, one has  
\begin{equation}\label{4.19}
   \widehat{\mathbf{N}}\big(F^{-1}(0), x\big) \cap \mathbf{B}_{\mathbb{X}^*}
   \subseteq 
   { \bigcup_{(u,v) \in \mathbf{B}\big((\bar{x},0),\epsilon\big) \cap \mathrm{gph}(F)}\left[ \tau\, \widehat{D}^*F(u,v)\big((1+\epsilon)\mathbf{B}_{\mathbb{Y}^*}\big)
    \right]}
   + \epsilon\,\mathbf{B}_{\mathbb{X}^*}.
\end{equation}

We next show that for any $(u,v) \in \mathrm{gph}(F)$, one has  
\begin{equation}\label{4.20}
   \widehat{D}^*F(u,v)\big((1+\epsilon)\mathbf{B}_{\mathbb{Y}^*}\big)
   \subseteq 
   [0,(1+\epsilon)\tau]\,\widehat{\partial}_K f(u) + \widehat{\partial}_K^{\infty} f(u).
\end{equation}
Granting this, it follows from \eqref{4.19} that \eqref{4.18} holds.  

Let $(u,v) \in \mathrm{gph}(F)$ and $x^* \in \widehat{D}^*F(u,v)\big((1+\epsilon)\mathbf{B}_{\mathbb{Y}^*}\big)$.  
Then there exists $y^* \in \mathbf{B}_{\mathbb{Y}^*}$ such that  
\[
   \big(x^*, -(1+\epsilon)y^*\big) \in \widehat{\mathbf{N}}\big(\mathrm{gph}(F), (u,v)\big).
\]
By \cref{lem4.2}, one has  
\begin{equation}\label{4.21}
   (x^*, -y^*) \in \widehat{\mathbf{N}}\big({\rm epi}_K(f), (u, f(u))\big).
\end{equation}

If $y^* = 0$, then \eqref{4.21} implies $x^* \in \widehat{\partial}_K^{\infty} f(u)$, and thus \eqref{4.20} holds.  

If $y^* \neq 0$, then from \eqref{4.21} we obtain  
\[
   \left( \frac{x^*}{(1+\epsilon)\|y^*\|}, -\frac{y^*}{\|y^*\|} \right)
   \in \widehat{\mathbf{N}}\big({\rm epi}_K(f), (u, f(u))\big),
\]
and consequently  
\[
   x^* \in (1+\epsilon)\|y^*\|\,\widehat{\partial}_K f(u) 
   \subseteq [0,\,1+\epsilon]\,\widehat{\partial}_K f(u)
   \quad \text{(since $\|y^*\| \leq 1$)}.
\]
This establishes \eqref{4.20}.  
\end{proof}

It is worth noting that all main calculus results in terms of Fr\'echet constructions 
generally hold only in a fuzzy form, and the parameter $\epsilon$ in \eqref{4.18} 
cannot, in general, be taken to be $0$.  
However, the following theorem shows that in the finite-dimensional setting it is 
possible to take $\epsilon=0$ by using limiting normal cones and subdifferentials.

\begin{theorem}\label{th4.3}
Let $\mathbb{X}$ be a finite-dimensional space, $\mathbb{Y}$ be an Asplund space, 
$K \subseteq \mathbb{Y}$ be a nonempty closed convex cone, and let 
$f \in \Lambda_K(\mathbb{X},\mathbb{Y})$ be such that the conic inequality 
$f(x) \leq_K 0$ is metrically subregular at 
$\bar{x} \in \mathbf{S}_K(f)$ with $f(\bar{x}) = 0$.  
Then there exist $\tau, \delta \in (0,+\infty)$ such that 
\begin{equation}\label{4.18a}
    \mathbf{N}(\mathbf{S}_K(f), x) \cap \mathbf{B}_{\mathbb{X}^*} 
    \subseteq [0,\tau]\,\partial_K f(x) + \partial_K^{\infty} f(x)
\end{equation}
for all $x \in \mathbf{S}_K(f) \cap \mathbf{B}(\bar{x},\delta)$ with $f(x)=0$.
\end{theorem}

\begin{proof}
Define $F(x) := f(x) + K$ for all $x \in \mathbb{X}$.  
Then $F \in \Gamma(\mathbb{X},\mathbb{Y})$ and 
$F^{-1}(0) = \mathbf{S}_K(f)$.  
By virtue of \cref{th3.3}, there exist $\tau, \delta \in (0,+\infty)$ such that 
\begin{equation}\label{4.23}
    \mathbf{N}(F^{-1}(0), x) \cap \mathbf{B}_{\mathbb{X}^*} 
    \subseteq \tau\, D^*F(x,0)(\mathbf{B}_{\mathbb{Y}^*}), 
    \quad \forall x \in \mathbf{B}(\bar{x},\delta) \cap F^{-1}(0).
\end{equation}

Let $x \in \mathbf{B}(\bar{x},\delta) \cap \mathbf{S}_K(f)$ be such that $f(x) = 0$.    {For \eqref{4.18a}, it suffices to}
 show that
\begin{equation}\label{4.24}
    D^*F(x,0)(\mathbf{B}_{\mathbb{Y}^*}) 
    \subseteq [0,1]\,\partial_K f(x) + \partial_K^{\infty} f(x).
\end{equation}

Take any $x^* \in D^*F(x,0)(\mathbf{B}_{\mathbb{Y}^*})$.  
Then there exists $y^* \in \mathbf{B}_{\mathbb{Y}^*}$ such that
\[
    (x^*,-y^*) \in \mathbf{N}(\mathrm{gph}(F),(x,0)) 
    = \mathbf{N}(\mathrm{epi}_K f,(x,f(x))).
\]

If $y^* = 0$, then $x^* \in \partial_K^{\infty} f(x)$, and \eqref{4.24} holds.  

If $y^* \neq 0$, then it follows from \cref{lem4.1} that 
$y^* \in K^{\infty,+}$, and hence
\[
    \left( \frac{x^*}{\|y^*\|}, -\frac{y^*}{\|y^*\|} \right) 
    \in \mathbf{N}(\mathrm{epi}_K f,(x,f(x))).
\]
This implies that
\[
    x^* \in \|y^*\|\,\partial_K f(x) 
    \subseteq [0,1]\,\partial_K f(x),
\]
so \eqref{4.24} follows.  
The proof is complete.
\end{proof}

\section*{Conclusion}

This paper investigates the metric subregularity of multifunctions and its connection to characterizations of Asplund spaces. For a closed multifunction between Asplund spaces, we have shown that metric subregularity implies certain exact and fuzzy inclusions, expressed respectively via limiting and Fr\'echet coderivatives as well as normal cones. Moreover, each of these inclusions, when applied to all closed multifunctions, serves as  {characterizations} of Asplund spaces.

As an application, we  {study} the metric subregularity of a conic inequality defined by a vector-valued function and a closed (not necessarily convex) cone. By employing limiting subdifferentials of the vector-valued function (with respect to the given closed cone), we obtained an exact inclusion implied by metric subregularity. When this exact inclusion is applied to all conic inequalities, it characterizes Asplund spaces. This leads to the insight that the result in \cite[Proposition~2]{LewisPang1998}---that a local error bound implies the $\mathbf{BCQ}$ for convex inequalities---is essentially a consequence of the Asplund property of finite-dimensional spaces.

In contrast, by using Fr\'echet subdifferentials (with respect to the same closed cone), we derived a fuzzy inclusion implied by metric subregularity. However, when applied to all conic inequalities, this fuzzy inclusion is only sufficient (and not necessary) for a space to be Asplund. This reveals a distinction between the Fr\'echet and limiting types of subdifferentials in their ability to characterize Asplund spaces.

Finally, we note that these two inclusions differ from the known characterizations of Asplund spaces obtained via limiting and Fr\'echet normal cones, as presented in \cite[Theorems~2.20, 2.22, and~2.30]{Mordukhovich}.

\bibliographystyle{plain}

\bibliography{WTY2025}

\begin{thebibliography}{10}

\bibitem{AC1988}
A.~Auslender and J.-P. Crouzeix.
\newblock Global regularity theorems.
\newblock {\em Math. Oper. Res.}, 13(2):243--253, 1988.

\bibitem{AusDanThi05}
D.~Aussel, A.~Daniilidis, and L.~Thibault.
\newblock Subsmooth sets: functional characterizations and related concepts.
\newblock {\em Trans. Amer. Math. Soc.}, 357(4):1275--1301, 2005.

\bibitem{Aze06}
D.~Az{\'e}.
\newblock A unified theory for metric regularity of multifunctions.
\newblock {\em J. Convex Anal.}, 13(2):225--252, 2006.

\bibitem{BB1993}
H.~H. Bauschke and J.~M. Borwein.
\newblock On the convergence of von neumann's alternating projection algorithm
  for two sets.
\newblock {\em Set-Valued Analysis}, 1(2):185--212, 1993.

\bibitem{BB1996}
H.~H. Bauschke and J.~M. Borwein.
\newblock On projection algorithms for solving convex feasibility problems.
\newblock {\em SIAM Rev.}, 38(3):367--426, 1996.

\bibitem{BD2}
J.V. Burke and S.~Deng.
\newblock Weak sharp minima revisited. {II}. {A}pplication to linear regularity
  and error bounds.
\newblock {\em Math. Program.}, 104(2-3, Ser. B):235--261, 2005.

\bibitem{C}
F.~H. Clarke.
\newblock {\em Optimization and Nonsmooth Analysis}.
\newblock Wiley, 1983.

\bibitem{De}
F.~Deutsch.
\newblock {\em The role of conical hull intersection property in convex
  optimization and approximation}.
\newblock in Approximation Theory IX, Vol. I: Theoretical Aspects (edited by
  C.K. Chui and L.L. Schumaker), Vanderbilt University Press, Nashville, TN,
  pp. 105-112.

\bibitem{DGZ1993}
R.~Deville, G.~Godefroy, and V.~Zizler.
\newblock {\em Smoothness and Renorming in Banach Spaces}.
\newblock Pitman Monogr. Surveys Pure Appl. Math. 64, Longman, Harlow,, 1993.

\bibitem{Dolecki1982}
S.~Dolecki.
\newblock Tangency and differentiation, some applications of convergence
  theory.
\newblock {\em Ann. Math. Pura Appl.}, 130:223–255, 1982.

\bibitem{DL}
A.~L. Dontchev, A.~S. Lewis, and R.~T. Rockafellar.
\newblock The radius of metric regularity.
\newblock {\em Trans. Amer. Math. Soc.}, 355(2):493--517, 2003.

\bibitem{DonRoc2004}
A.~L. Dontchev and R.~T. Rockafellar.
\newblock Regularity and conditioning of solution mappings in variational
  analysis.
\newblock {\em Set-Valued Analysis}, 12(1-2):79--109, 2004.

\bibitem{DonRoc09}
A.~L. Dontchev and R.~T. Rockafellar.
\newblock {\em Implicit functions and solution mappings : a view from
  variational analysis}.
\newblock Springer New York, 2009.

\bibitem{F}
M.~Fabian.
\newblock Subdifferentiability and trustworthiness in the light of a new
  variational principle of borwein and preiss.
\newblock {\em Charles University in Prague}, 30(2), 1989.

\bibitem{FM1998}
M.~Fabian and B.~S. Mordukhovich.
\newblock Nonsmooth characterizations of asplund spaces and smooth variational
  principles.
\newblock {\em Set-Valued Analysis}, 6:381--406, 1998.

\bibitem{FLN2010}
D.~H. Fang, C.~Li, and K.~F. Ng.
\newblock Constraint qualifications for extended {F}arkas's lemmas and
  lagrangian dualities in convex infinite programming.
\newblock {\em SIAM Journal on Optimization}, 20(3):1311--1332, 2010.

\bibitem{Graves1950}
L.~M. Graves.
\newblock Some mapping theorems.
\newblock {\em Duke Math. J.}, 17:111–114, 1950.

\bibitem{HenrionJourani2002-SIAM}
R.~Henrion and A.~Jourani.
\newblock Subdifferential conditions for calmness of convex constraints.
\newblock {\em SIAM J. Optim.}, 13(2):520--534, 2002.

\bibitem{HenrionJouraniOutrata2002-SIAM}
R.~Henrion, A.~Jourani, and J.~Outrata.
\newblock On the calmness of a class of multifunctions.
\newblock {\em SIAM J. Optim.}, 13:603–618, 2002.

\bibitem{HenrionOutrata2004-MP}
R.~Henrion and J.~Outrata.
\newblock Calmness of constraint systems with applications.
\newblock {\em Math. Program.}, 104:437–464, 2005.

\bibitem{HenOut2005}
R.~Henrion and J.~V. Outrata.
\newblock Calmness of constraint systems with applications.
\newblock {\em Mathematical Programming}, 104(2):437--464, 2005.

\bibitem{H-UL1993}
J.-B. Hiriart-Urruty and C.~Lemarechal.
\newblock {\em Convex Analysis and Minimization Algorithms I Fundamentals}.
\newblock Springer-Verlag, Berlin, 1993.

\bibitem{Hoffman1952}
A.~J. Hoffman.
\newblock On approximate solutions of systems of linear inequalities.
\newblock {\em J. Research Nat. Bur. Standards}, 49:263--265, 1952.

\bibitem{Ioffe1989}
A.~D. Ioffe.
\newblock Approximate subdifferentials and applications 3: the metric theory.
\newblock {\em Mathematika}, 36(1):1--38, 1989.

\bibitem{Ioffe2000}
A.~D. Ioffe.
\newblock Metric regularity and subdifferential calculus.
\newblock {\em Russian Math. Surveys}, 55:501–558, 2000.

\bibitem{ioffe-JAMS-1}
A.~D. Ioffe.
\newblock Metric regularity---a survey {P}art 1. {T}heory.
\newblock {\em J. Aust. Math. Soc.}, 101(2):188--243, 2016.

\bibitem{ioffe-JAMS-2}
A.~D. Ioffe.
\newblock Metric regularity---a survey {P}art {II}. {A}pplications.
\newblock {\em J. Aust. Math. Soc.}, 101(3):376--417, 2016.

\bibitem{Ioffe2017}
A.~D. Ioffe.
\newblock {\em Variational Analysis of Regular Mappings: Theory and
  Applications}.
\newblock Springer Monographs in Mathematics. Springer, 2017.

\bibitem{IoffeOutrata2008-SVA}
A.~D. Ioffe and J.~V. Outrata.
\newblock On metric and calmness qualification conditions in subdifferential
  calculus.
\newblock {\em Set-Valued Anal.}

\bibitem{IP1996}
A.~D. Ioffe and J-P. Penot.
\newblock Subdifferentials of performance functions and calculus of
  coderivatives of set-valued mappings.
\newblock {\em Serdica Math J}, 22(22):359--384, 1996.

\bibitem{K2}
A.~Y. Kruger.
\newblock Error bounds and metric subregularity.
\newblock {\em Optimization}, 64(1):49--79, 2015.

\bibitem{LewisPang1998}
A.~S. Lewis and J-S. Pang.
\newblock Error bounds for convex inequality systems.
\newblock In {\em Generalized Convexity, Generalized Monotonicity: Recent
  Results ({L}uminy, 1996)}, volume~27 of {\em Nonconvex Optim. Appl.}, pages
  75--110. Kluwer Acad. Publ., Dordrecht, 1998.

\bibitem{LNP2008}
C.~Li, K.~F. Ng, and T.~K. Pong.
\newblock Constraint qualifications for convex inequality systems with
  applications in constrained optimization.
\newblock {\em SIAM Journal on Optimization}, 19:163--187, 2008.

\bibitem{Li}
W.~Li.
\newblock {A}badie's constraint qualification, metric regularity, and error
  bounds for differentiable convex inequalities.
\newblock {\em SIAM Journal on Optimization}, 7(4):966--978, 1997.

\bibitem{LNS2000}
W.~Li, C.~Nahak, and I.~Singer.
\newblock Constraint qualifications for semi-infinite systems of convex
  inequalities.
\newblock {\em SIAM Journal on Optimization}, 11(1):31--52, 2000.

\bibitem{Lyusternik1934}
L.~A. Lyusternik.
\newblock On conditional extrema of functionals.
\newblock {\em Math. Sbornik}, 41:390–401, 1934.

\bibitem{Mordukhovich}
B.~S. Mordukhovich.
\newblock {\em Variational Analysis and Generalized Differentiation I}.
\newblock Springer, New York, NY, 2006.

\bibitem{MS}
B.~S. Mordukhovich and Y.~Shao.
\newblock Nonsmooth sequential analysis in asplund spaces.
\newblock {\em Transactions of the American Mathematical Society},
  348:1235--1280, 1996.

\bibitem{MW2000}
B.~S. Mordukhovich and B.~Wang.
\newblock On variational characterizations of {A}splund spaces.
\newblock In {\em Constructive, experimental, and nonlinear analysis
  ({L}imoges, 1999)}, volume~27 of {\em CRC Math. Model. Ser.}, pages 245--254.
  CRC, Boca Raton, FL, 2000.

\bibitem{NT2001}
H.~V. Ngai and M.~Th\'era.
\newblock Metric inequality, subdifferential calculus and applications.
\newblock {\em Set Valued Analysis}, 9(1-2):187--216, 2001.

\bibitem{Pen13}
J.~P. Penot.
\newblock {\em Calculus without Derivatives}.
\newblock Graduate Texts in Mathematics, vol. 266. Springer, New York, 2013.

\bibitem{Phelps}
R.~R. Phelps.
\newblock {\em Convex Functions, Monotone Operators and Differentiability}.
\newblock Lecture Notes in Math. 1364, Springer, New York, 1989.

\bibitem{Robinson1975}
S.M. Robinson.
\newblock An application of error bounds for convex programming in a linear
  space.
\newblock {\em SIAM J. Control}, 13:271--273, 1975.

\bibitem{thibault}
L.~Thibault.
\newblock {\em Unilateral Variational Analysis in Banach Spaces}.
\newblock World Scientific, 2023.

\bibitem{WTY2025-JCA}
Z.~Wei, M.~Th\'era, and J.-C. Yao.
\newblock On error bounds for inequalities in {A}splund spaces.
\newblock {\em Journal of Convex Analysis, {\it to appear}}.

\bibitem{WTY2024-SVVA}
Z.~Wei, M.~Th\'era, and J.-C. Yao.
\newblock Subtransversality and strong {CHIP} of closed sets in {A}splund
  spaces.
\newblock {\em Set-Valued Var. Anal.}, 32(23), 2024.

\bibitem{WY}
Z.~Wei and J.-C. Yao.
\newblock On constraint qualifications of a nonconvex inequality.
\newblock {\em Optimization Letters}, 12:1117--1139, 2018.

\bibitem{Zheng2022-MOR}
X.~Y. Zheng.
\newblock Slater condition for tangent derivatives.
\newblock {\em Math. Ope. Res.}, 47(4):3282--3303, 2022.

\bibitem{ZN2004}
X.~Y. Zheng and K.~F. Ng.
\newblock Metric regularity and constraint qualifications for convex
  inequalities on {B}anach spaces.
\newblock {\em SIAM Journal on Optimization}, 14(3):757--772, 2004.

\bibitem{ZN2007}
X.~Y. Zheng and K.~F. Ng.
\newblock Metric subregularity and constraint qualifications for convex
  generalized equations in {B}anach spaces.
\newblock {\em SIAM Journal on Optimization}, 18(2):437--460, 2007.

\bibitem{ZN}
X.~Y. Zheng and K.~F. Ng.
\newblock Linear regularity for a collection of subsmooth sets in {B}anach
  spaces.
\newblock {\em SIAM Journal on Optimization}, 19(1):62--76, 2008.

\bibitem{ZhengNg2010-SIAM}
X.~Y. Zheng and K.~F. Ng.
\newblock Metric subregularity and calmness for nonconvex generalized equations
  in {B}anach spaces.
\newblock {\em SIAM J. Optim.}, 20(5):2119--2136, 2010.

\bibitem{ZhengNg2012-NA}
X.~Y. Zheng and K.~F. Ng.
\newblock Metric subregularity for proximal generalized equations in {H}ilbert
  spaces.
\newblock {\em Nonlinear Anal.}, 75(3):1686--1699, 2012.

\bibitem{ZN2019-ESAIM}
X.~Y. Zheng and K.~F. Ng.
\newblock Stability of error bounds for conic subsmooth inequalities.
\newblock {\em ESAIM Contr. Optim. CA}, 25:55, 2019.

\bibitem{Zalinescu}
C.~Z\u{a}linescu.
\newblock {\em Convex {A}nalysis in {G}eneral {V}ector {S}paces}.
\newblock World Scientific Publishing Co., Inc., River Edge, NJ, 2002.

\end{thebibliography}

\end{document}